\documentclass[11pt,a4paper]{article}

\usepackage{subfigure}
\usepackage{graphicx}
\usepackage{epstopdf}
\usepackage{amssymb}
\usepackage[english]{babel}
\usepackage{amsmath}
\usepackage{amsfonts}
\usepackage{latexsym}
\usepackage{color}
\usepackage{amsthm}
\usepackage{bm}

\newtheorem{theorem}{\bf Theorem}[section]

\newtheorem{proposition}{\bf Proposition}[section]

\newtheorem{rmk}{\bf Remark}[section]
\newcommand{\be}{\begin{equation}}
\newcommand{\ee}{\end{equation}}

\def\RR{\mathbb{R}}

\def\NLk{N_k}
\def\mF{m_F}
\def\vk{\bar v^k}
\def\vl{\bar v^\ell}
\def\fs{f_{\infty}}
\def\gs{g_\infty}

\begin{document}

%%%% Article title to be placed here
\title{Boltzmann games in heterogeneous \\consensus dynamics}

\author{Giacomo Albi \thanks{Department of Computer Sciences, University of Verona, Str. Le Grazie 15, 37134, Verona, Italy ({\tt giacomo.albi@univr.it})} \and Lorenzo Pareschi \thanks{Department of Mathematics and Computer Science, University of Ferrara, Via Machiavelli 35, 44121 Ferrara, Italy ({\tt lorenzo.pareschi@unife.it}).} \and
Mattia Zanella\thanks{Department of Mathematical Sciences, Politecnico di Torino, Corso Duca degli Abruzzi 24, Torino, Italy ({\tt mattia.zanella@polito.it}).}}

%%%% Abstract text to be placed here %%%%%%%%%%%%
\maketitle

\begin{abstract}
We consider a constrained hierarchical opinion dynamics in the case of leaders' competition and with complete information among leaders. Each leaders' group tries to drive the followers' opinion towards a desired state accordingly to a specific strategy. By using the Boltzmann--type control approach we analyze the best--reply strategy for each leaders' population. Derivation of the corresponding Fokker-Planck model permits to investigate the asymptotic behaviour of the solution. Heterogeneous followers populations are then considered where the effect of knowledge impacts the leaders' credibility and modifies the outcome of the leaders' competition. \\


\end{abstract}
%%%%%%%%%%%%%%%%%%%%%%%%%%%

%%%%%%%%%% Insert the texts which can accomdate on firstpage in the tag "fmtext" %%%%%

\section{Introduction}
Control methodologies for multi--agent systems gained in recent years an increasing interest due to their effectiveness in modelling efficient techniques to force systems toward a given benchmark configuration \cite{ACFK,APTZ,FPR,HSP,HZ}. The control of emergent behaviour has been studied at the level of the agents' dynamics as at the level of the kinetic and hydrodynamic equations \cite{APZa,BFK,HSP,DLR}. The general setting consists in a microscopic model described by a system of ODEs where the evolution of the constrained state of each agent minimizes a given cost functional and is influenced by the collective behaviour of all the other agents.

When the number of agents is very large a direct solution of the microscopic problem becomes prohibitive and the behaviour of the system can be approximated by kinetic (or hydrodynamic) partial differential equations \cite{CFRT,CFTV,CPT,CT}.

The present work focus on hierarchical opinion controls based on several populations of leaders which compete in order to drive the opinion of the followers towards a desired state. Each leader acts via the minimization of a suitable cost function characterizing its strategy.

Unlike the approach proposed in \cite{APZa,DMPW}, in the present setting both the opinion variable and the control strategies of followers and leaders are always known to each \emph{player}, which makes the best possible decision taking into account what the opponents are doing, whereas the dynamics of the followers are influenced by the leaders and by the other followers. Hence, the introduced opinion model can be seen as a \emph{game} in which the opponents are the opinion leaders which try to maximize their impact over the followers' population, i.e. in other words the overall dynamics can be seen as a differential game \cite{C96,R65}. On the basis of a microscopic model we propose a Boltzmann--type model predictive control (MPC) approach following the ideas in \cite{AHP,APTZ,APZa}. By means of this technique we explicit the best reply dynamics at the binary level and we study the related kinetic description for a large number of agents. 

Next, in order to have a more realistic description of the opinion formation process, we introduce the followers' heterogeneity. In particular, we consider the impact of knowledge in the opinion dynamics by inducing resistance to social influence \cite{BOLRRF,HW}. The followers' behaviour is then described by different processes: the opinion dynamics depend on a knowledge based compromise process and on the interaction with each leaders' population, further, the evolution of the knowledge depends on a social background in which individuals may gain knowledge from more skilled agents \cite{BLW,BT,PT1,PT2,PVZ}. Numerical tests permit to show the effect of heterogeneity over the leaders' competition outcome.

The rest of the paper has the following structure: in Section \ref{sect:micro} we first introduce the constrained microscopic problem for an arbitrary number of leaders' populations, and next we exploit the model predictive approach to derive explicit binary interaction rules which embed the best--reply strategy for each leader.  The Boltzmann--type formulation of the binary game is given in Section \ref{sect:boltzmann}, the evolution of macroscopic quantities like the mean opinion and its variance are here investigated. Explicit asymptotic distributions are also reported in Section \ref{sect:FP} in the so--called quasi--invariant limit. In Section \ref{sect:FP_het} a multivariate model for heterogeneous followers based on the level of knowledge is proposed. Numerical results are shown in Section \ref{sect:num} and, beside confirming the theoretical analysis, show how different levels of knowledge in the followers' population may lead to different behaviours in the competition between leaders.

\section{A microscopic differential game with multiple leaders}\label{sect:micro}
We are interested in the opinion formation process of a population of followers influenced by the action of various leaders. The strategy of the leaders is driven by a control term characterised by the leaders' radical and populist attitudes \cite{AHP,APZa}. The first attitude aims at forcing the agents toward a given desired opinion, whereas the second attitude at keeping a position close to the mean opinion of the overall population. In contrast to \cite{APZa}, here the leaders compete with each others in agreement with their strategies thus giving rise to a differential game.

\subsection{Microscopic mean--field modelling}
 Let us consider a population of $N_F$ followers with opinion $w_i\in \mathcal I$, where $\mathcal I = [-1,1]$ and $M$ groups of leaders $v^k_{h}$ with opinion $h=1,\dots,\NLk$ for the population in the group $k = 1,\dots,M$. 

We consider the dynamics described by the following equations
\be\begin{split}\label{eq:dyn}
& \dot w_i = \dfrac{1}{N_F}\sum_{j=1}^{N_F} P(w_i,w_j)(w_j-w_i)+ \sum_{\ell=1}^M\Bigg(\dfrac{1}{N_\ell}\sum_{h=1}^{N_\ell}R^\ell(w_i,v_{h}^\ell)(v_{h}^\ell-w_i)\Bigg)\\
& \dot{v}_{h}^k = \dfrac{1}{\NLk} \sum_{p = 1}^{\NLk}S^k(v_{h}^k,v_p^k)(v_p^k-v_{h}^k)+\sum_{\ell=1}^M u^\ell
\end{split}\ee
with given initial conditions, $i=1,\dots,N_F$, $k=1,\dots,M$ and $h = 1,\dots,\NLk$. 
In \eqref{eq:dyn} we introduced the compromise functions $P(\cdot,\cdot)\in[0,1]$ and $R^k(\cdot,\cdot)\in[0,1]$ measuring respectively the strength of interaction between followers and the strength of interaction between followers and leaders of the $k$th-population. Further, $S^k(\cdot,\cdot)\in[0,1]$ measures the strength of interaction between leaders of population $k$, which we will assume to be {\em symmetric}, i.e. $S^k(v,w)=S^k(w,v)$ for every $k=1,\ldots,M$.

The controls $u^k$ characterize the strategies of the leaders of the system and are solution of the control problems
\be
u^k = \textrm{arg}\min_{u^k\in\mathcal U} J^k(u^k,{\bf u} ^k_{-};\textbf{w},\textbf{v}),\qquad k=1,\dots,M
\ee
where $\mathcal U$ is the set of admissible controls and the cost functional $ J^k(u^k,{\bf u} ^k_{-};\textbf{w},\textbf{v})$ is defined as follows
\be\begin{split}\label{eq:J}
 &J^k(u^k,{\bf u} ^k_{-};\textbf{w},\textbf{v}) =\\
 &\quad \dfrac{1}{2}\displaystyle \int_{0}^T \left[ \dfrac{\psi^k}{\NLk}\sum_{h=1}^{\NLk}|v_{h}^k-\vk|^2 + \dfrac{\mu^k}{\NLk} \sum_{h=1}^{\NLk}|v_{h}^k-\mF|^2 + \nu^k |u^k|^2 \right]ds,
\end{split}\ee
being ${\bf u} ^k_{-}$ the vector of the other leaders, $k$ excluded.
In \eqref{eq:J} the terms $\textbf{w},\textbf{v}$ are the vectors of the opinions of the followers and all the leaders respectively, $T$ is the final time horizon of the optimization procedure, $\vk\in[-1,1]$ are desired opinions of the $k$th group of leaders. We also indicated with $\mF=\mF(t)\in[-1,1]$ the mean opinion of the followers at time $t\geq0$, and equivalently we will indicate with  $m_L^k=m_L^k(t)\in[-1,1]$ the mean opinion expressed by the $k$th population of leaders, both quantities are computed as follows
\begin{equation}\label{eq:MFave}
\mF(t) = \dfrac{1}{N_F}\sum_{j=1}^{N_F} w_j(t), \quad m_L^k(t) = \dfrac{1}{N_k}\sum_{h=1}^{N_k} v^k_h(t).
\end{equation}
The parameter $\nu^k>0$ is a regularization term representing the importance of the control term in the overall dynamics. The coefficients $\psi^k,\mu^k$, weight two possible behaviours: $\psi^k$ is related to the {\em radical} attitude of maintaining a certain desired position $\vk$, whereas $\mu^k$ weights the tendency of getting closer to the average opinion of the followers' population, $\mF$, we refer to this behaviour as {\em populist}. We further assume that the leaders' strategies are a convex combination of these two behaviours, i.e. $\psi^k+\mu^k=1$.

\begin{rmk}
A further generalization of the model consists in considering an additional parameter $\theta_{k\ell}\in[0,1]$ weighting the influence of each strategy $u^\ell$, $\ell=1,\ldots,M$ on the $k$th leaders' population, namely we have 
\be\begin{split}\label{eq:dyntheta}
%& \dot w_i = \dfrac{1}{N_F}\sum_{j=1}^{N_F} P(w_i,w_j)(w_j-w_i)+ \sum_{\ell=1}^M\Bigg(\dfrac{1}{N_\ell}\sum_{h=1}^{N_\ell}R^\ell(w_i,v_{h}^\ell)(v_{h}^\ell-w_i)\Bigg)\\
& \dot{v}_{h}^k = \dfrac{1}{\NLk} \sum_{j = 1}^{\NLk}S^k(v_{h}^k,v_j^k)(v_j^k-v_{h}^k)+\sum_{\ell=1}^M\theta_{k\ell} u^\ell.
\end{split}\ee
Parameters $\theta_{k\ell}$ can describe situations where the knowledge of other leaders' strategies is incomplete or biased by uncertainties.
Note that for $ \bm{\Theta} =[\theta_{k\ell}] = Id_{M\times M}$ model \eqref{eq:dyn} decouples in a system of leaders where every leaders' population has its own strategy, as originally presented in \cite{APZa}.
\end{rmk}

\subsection{Binary game approximation}
To reduce the model's complexity we consider first a discrete version of \eqref{eq:dyn} taking into account binary interactions. Hence, we introduce the time discretization $ 0=t^0 <t^1<\ldots<t^{N_{t}}=T$ with $t^n = n\Delta t$, the parameter $\alpha = \Delta t/2$, and we approximate the differential system \eqref{eq:dyn} via pairwise interactions. We obtain the following binary dynamics among each population of leaders $k=1,\dots,M$ 
\be\begin{split}\label{eq:bin_leaders}
v_h^{k,n+1}=v_h^{k,n} + \alpha S^k(v_h^{k,n},v_p^{k,n})(v_p^{k,n}-v_h^{k,n}) + 2\alpha\sum_{\ell=1}^M u^{\ell,n},  \\
v_p^{k,n+1}=v_p^{k,n} + \alpha S^k(v_p^{k,n},v_h^{k,n})(v_h^{k,n}-v_p^{k,n}) + 2\alpha\sum_{\ell=1}^M u^{\ell,n},  \\
\end{split}\ee
where the controls $u^{k,n}$ are now determined by solving  the optimization problem on the binary dynamics. The cost functional over the binary interactions is given by 
\be
\begin{split}\label{eq:disc_funcJu0}
&J^{k,n}(u^k,{\bf u} ^k_{-}) : =\\
&\quad \int_{t^n}^{t^{n+1}} \frac{1}{2N_k} \sum_{h=1}^{N_k}\left(\psi^k|v^{k}_{h}-\vk|^2+\mu^k|v_{h}^{k}-\mF|^2+\nu^k |u^{k}|^2\right)\ dt, 
\end{split}
\ee
thus the control variables $u^{k,n}$ for $k=1,\ldots,M$ are solution to the following coupled system
\be\begin{split}\label{eq:diff_game}
u^{1,n}& = \textrm{arg}\min J^{1,n}(u^{1,n},{\bf u}_{-}^{1,n})\\
u^{2,n} &= \textrm{arg}\min J^{2,n}(u^{2,n},{\bf u}_{-}^{2,n})\\
&\,\,\vdots\\
u^{M,n} &= \textrm{arg}\min J^{M,n}(u^{M,n},{\bf u}_{-}^{M,n}).
\end{split}\ee

Note that this type of problem has the structure of a differential game, and therefore the controls  should be determined as an equilibrium of the competition process, \cite{R65, C96}. 
In what follows we show how to determine an approximation of the solution of this problem through the derivation of a best-reply approach \cite{AHP, APZa,DHR}. 

\subsubsection{Binary best-reply strategies}
A direct method to embed the explicit formulation of best-reply strategies in the dynamics given by \eqref{eq:diff_game} is to introduce a system of Euler-Lagrange equations, and to compute the minimizers of the resulting equations, see also \cite{DHR,ACFK} for further details. Similarly to instantaneous strategies \cite{AHP,APZa}, we approximate the solution of \eqref{eq:diff_game} by introducing a semi-implicit time-discretization of \eqref{eq:disc_funcJu0} as follows
\be\begin{split}\label{eq:disc_funcJu}
&\tilde J^{k,n}(u^{k,n},{\bf u} ^{k,n}_{-})=\\
&\quad \frac{1}{2N_k} \sum_{h=1}^{N_k}\left(\psi^{k}|v^{k,n+1}_h-\vk|^2+\mu^{k}|v^{k,n+1}_h-\mF^n|^2 +\nu^{k} |u^{k,n}|^2\right),
\end{split}\ee
where we assume to evaluate $\mF(t)$ in \eqref{eq:MFave} at time $t^n$.
Hence, by computing the minimizers of  \eqref{eq:disc_funcJu} w.r.t. $u^{k,n}$ for all $k=1,\ldots,M$, we obtain 
\begin{align*}
&D_{u^{k,n}} \tilde J^{k,n}(u^{k,n}) =\\
&\quad \frac{1}{N_k} \sum_{h=1}^{N_k}\left(2\alpha \psi^{k}(v^{k,n+1}_h-\vk)+2\alpha\mu^{k}(v^{k,n+1}_h-\mF^n) +\nu^{k}u^{k,n}\right)=
\\
&\quad \frac{2\alpha }{N_k} \sum_{h=1}^{N_k} \left(\psi^{k}(v^{k,n}_h-\vk)+\mu^{k}(v^{k,n}_h-\mF^n)\right) + \frac{2\alpha^2}{N_k} \sum_{h=1}^{N_k} S^k(v_h^{k,n},v_p^{k,n})(v_p^{k,n}-v_h^{k,n})
\\
&\quad+ 4\alpha^2\sum_{\ell=1}^M u^{\ell,n}+ \nu^{k} u^{k,n},
%&\quad = 2\alpha \left(\psi^{k}(m^{k,n}_L-\vk)+\mu^{k}(m^{k,n}_L-\mF^n)\right) + 4\alpha^2\sum_{\ell=1}^M u^{\ell,n}+ \nu^{k} u^{k,n}
\end{align*}
where we substitute $v^{k,n+1}$  with its explicit formulation \eqref{eq:bin_leaders}. Note also that, under symmetricity assumption of $S^k(\cdot,\cdot)$, the contribution of the sum of the interactions $S^k(v_h,v_p)(v_p-v_h)$ vanishes, since it is composed by binary terms such that for every pair $(h,p)$ we have 
\[\begin{split}
&S^k(v_h^{k},v_p^{k})(v_p^{k}-v_h^{k}) + S^k(v_p^{k},v_h^{k})(v_h^{k}-v_p^{k})=\\
&\qquad \left(S^k(v_h^{k},v_p^{k})-S^k(v_h^{k},v_p^{k})\right)(v_p^{k}-v_h^{k})=0. 
\end{split}\]
Thus, the previous expression reduces to 
\begin{align*}
&D_{u^{k,n}} \tilde J^{k,n}(u^{k,n}) =\\
&\qquad 2\alpha \left(\psi^{k}(m^{k,n}_L-\vk)+\mu^{k}(m^{k,n}_L-\mF^n)\right) + 4\alpha^2\sum_{\ell=1}^M u^{\ell,n}+ \nu^{k} u^{k,n},
\end{align*}
where $m_L^k(t)$ are the average opinion of the $k$th-population of leaders at time $t$. 

We find the critical points of the functional imposing the last relation to be zero, 
\be\begin{split}\label{eq:bin_ctrlAm}
\left(\nu^k + 4\alpha^2\right)u^{k,n} &+ 4\alpha^2\sum_{\ell\neq k}u^{\ell,n} =\\
& 2\alpha \left(\psi^{k}\left(\vk-m^{k,n}_L\right)+\mu^{k}\left(\mF^n-m^{k,n}_L\right)\right),
\end{split}\ee
where the second member  of the expression contains the convex combination of the relaxation towards the desired state $\vk$ and the average of the followers' opinions, $\mF$.  Denoting by $F_\alpha^{k,n}$ the right-hand side of \eqref{eq:bin_ctrlAm} , and introducing the parameter $\beta^{k}$ as follows
\begin{align}\label{eq:strategym}
F^{k,n}:= \psi^{k}\vk+\mu^{k}\mF^n-m^{k,n}_L,\quad \beta^{k} := \dfrac{4\alpha^2}{\nu^{k} + 4\alpha^2},
\end{align}
 we can write system \eqref{eq:bin_ctrlAm},  for $k=1,\ldots,M$, as  
\begin{align}\label{eq:bin_ctrlBm}
& u^{k,n} +\beta^{k} \sum_{\ell \neq k}^M u^{\ell,n}=\frac{\beta^k}{2\alpha}F^{k,n},
\end{align}
or equivalently in vector notation as
\begin{align}\label{syscontrolm}
\begin{bmatrix}
1& \beta^1& \beta^1&\ldots &\beta^1\\
\beta^2& 1& \beta^2&\ldots &\beta^2\\
\vdots&\vdots&\vdots&&\vdots\\
% \beta^{M-1}& \beta^{M-1}& \beta^{M-1}\ldots&1 & \beta^{M-1}\\
 \beta^M& \beta^M&\beta^M&\ldots &1
\end{bmatrix}
\begin{bmatrix}
u^{1,n}\\
u^{2,n}\\
\vdots\\
u^{M,n}
\end{bmatrix}
=
\frac{1}{2\alpha}
\begin{bmatrix}
\beta^1F^{1,n}\\
\beta^2F^{2,n}\\
\vdots\\
\beta^MF^{M,n}
\end{bmatrix}.
\end{align}
The existence and uniqueness of the solution can be assured under the following sufficient condition
\begin{proposition}
System  \eqref{eq:bin_ctrlBm} admits an unique solution if the following condition
\begin{align}\label{suffcondm}
\nu^k> 4(M-2)\alpha^2,
\end{align}
holds true for every $k=1,\ldots M$.
\end{proposition}
This condition is obtained by imposing the strictly diagonal dominance of the matrix of system \eqref{syscontrolm}. Moreover, recalling the equivalence between $\alpha$ and $\Delta t/2$, condition \eqref{suffcondm} may be seen as a bound on the time-stepping of the binary interaction, i.e. 
\begin{align}\label{suffcond2}
\Delta t<\sqrt{\nu^k/(M-1)},\quad k=1,\ldots,M.
\end{align}
Under this condition we can invert system \eqref{eq:bin_ctrlBm} and the control $u^k$ is obtained as follows 
\begin{align}\label{eq:bin_ctrlB_solm}
u^{k,n} =\frac{1}{2\alpha}\sum_{\ell=1}^M\beta^{\ell}B^\alpha_{k,\ell} F^{\ell,n},
\end{align}
where the coefficients ${\bf B}^\alpha = [B^\alpha_{k,\ell}]$ define the inverse matrix of system  \eqref{syscontrolm}. Since the control applied to each leaders is the sum over $k$ of the strategies $u^k$, in order to compute the control term is sufficient to sum over $k$ relation \eqref{eq:bin_ctrlB_solm} as follows
\begin{align}\label{eq:bin_ctrlB_sol_avem}
\sum_{k=1}^Mu^{k,n}  = \frac{1}{2\alpha}\sum_{\ell=1}^M\beta^{\ell}\bar B^\alpha_{\ell} \left(\psi^{\ell}\vl+\mu^{\ell}\mF^n-m_L^{\ell,n}\right),
\end{align}
where we substituted the explicit version of $F^{\ell,n}$ \eqref{eq:strategym}, and with $\bar B^\alpha_\ell$ defined as
\be\label{eq:sumB}
\bar B^\alpha_{\ell} := \sum_{k=1}^M B^\alpha_{k\ell}.
\ee
Hence, to obtain the strategy $\sum_{k=1}^Mu^{k,n}$ it is sufficient to compute once ${\bf B}^\alpha$, and to update at each time step the binary interaction \eqref{eq:bin_leaders} embedding the control \eqref{eq:bin_ctrlB_sol_avem}.
\\

In the following sections we will show how these binary constrained dynamics can be embedded in Boltzmann-type equations, and how to derive a consistent system of Fokker-Planck-type equations for the mean-field system \eqref{eq:dyn}.
\begin{rmk}
In the simplified setting in which the penalization parameters are uniformly equal, $\nu^1=\nu^2=\ldots=\nu^M = \nu$, since the dynamics of the leaders are controlled by the sum over all the strategies, the previous system is easily solved.
Then in this case $\beta^k=\beta$ for every $k=1,\ldots,M$ and summing over $k$ equation \eqref{eq:bin_ctrlBm} we have
\begin{align}\label{eq:bin_avesystem}
&(1-\beta)\bar u^{n} +\beta M \bar u^{n}=\frac{\beta}{2\alpha}\sum_{\ell=1}^M \left(\psi^{\ell}\vl+\mu^{\ell}\mF^n-m_L^{\ell,n}\right),
\end{align}
where we defined $\bar u= \sum_{k=1}^M u^k$, which restitutes the following control 
\begin{align}\label{eq:bin_avesol}
&\bar u^{n}  = \frac{\beta}{2\alpha(1+(M-1)\beta)}\sum_{\ell=1}^M \left(\psi^{\ell}\vl+\mu^{\ell}\mF^n-m_L^{\ell,n}\right).
\end{align}
%where we inserted $F^{\ell,n}_\alpha$ in its explicit formulation  \eqref{eq:strategy}. %This situation can be interpreted the case in which every leader has the same amount of resources, or equivalently the cost of implementing its own strategy is the equal for every population of leaders.
\end{rmk}

\begin{rmk}\label{rmk:localfunc}
A different approach of deriving binary strategies is to consider a different functional with respect to \eqref{eq:disc_funcJu}, which rules only locally the binary discrete interactions of the leaders \cite{AHP,APZa}, as follows
\be\begin{split}\label{eq:disc_funcJuM}
&\tilde J^{k,n}(u^{k,n},{\bf u} ^{k,n}_{-})=\\
&\quad\frac{1}{2} \sum_{r\in\{h,p\}}\left(\psi^{k}|v^{k,n+1}_r-\vk|^2+\mu^{k}|v^{k,n+1}_r-\mF^n|^2 +\nu^{k} |u^{k,n}|^2\right).
\end{split}\ee
Thus computing the minimizers of  \eqref{eq:disc_funcJuM} w.r.t. $u^{k,n}$ for all $k=1,\ldots,M$, leads to the following system 
\begin{align}\label{eq:bin_ctrlB}
&u^{k,n} +\beta^{k} \sum_{\ell\neq k}^M u^{\ell,n}=\frac{\beta^k}{2\alpha}\left(\psi^{k}\vk+\mu^{k}\mF^n-\frac{v_h^{k,n}+v_p^{k,n}}{2}\right),
\end{align}
which shares the same structure of system \eqref{eq:bin_ctrlBm}, with the difference that the right-hand side 
includes the relaxation towards the desired state $\psi^{k}\vk+\mu^{k}\mF^n$ of the {\em local average} of two leaders' opinions defined as follows 
\be\label{eq:locave}
\hat{m}^{k,n}_L: = \frac{v_p^{k,n}+v_h^{k,n}}{2}.
\ee
Thus under the same condition \eqref{suffcondm} we can compute the sum of the strategies as follows
\be\label{eq:bin_ctrlB_sol_locavem}
\sum_{k=1}^Mu^{k,n}  = \frac{1}{2\alpha}\sum_{\ell=1}^M\beta^{\ell}\bar B^\alpha_{\ell} \left(\psi^{\ell}\vl+\mu^{\ell}\mF^n- \frac{v_p^{\ell,n}+v_h^{\ell,n}}{2}\right).
\ee
%and not the global average ${m}^{k,n}_L$.
\end{rmk}

\section{The Boltzmann game}\label{sect:boltzmann}
In this section we consider a Boltzmann-type dynamics to describe the evolution of the binary game approximation \eqref{eq:bin_leaders}--\eqref{eq:disc_funcJu0}. Let $f=f(w,t)$ be the  density at time $t\ge 0$ for the agents with opinion $w\in I=[-1,1]$, which we assume to be normalized such that $\int_{ I} f(w,t)dv =1$. Furthermore, let $g^k(v,t)$ be the density of the leaders' population $k=1,\dots,M,$
\be\label{eq:rhop}
\int_{ I} g^k(v,t)dv = \rho^k\le 1,
\ee
thus, using the same notations as in Section \ref{sect:micro}, we define the average opinion of followers'  and leaders' population respectively
\begin{align}\label{eq:mF}
\mF(t):=\int_{ I} w f(w,t)dw,\qquad
m_L^k(t):= \frac{1}{\rho_k}\int_{ I} v g^k(v,t)dv.
\end{align}
In the following we will derive the kinetic description for the evolution of the densities $f$ and $g^k$ through classical methods of kinetic theory \cite{PT2,T, AHP, APZa}.

%%%%%%%%%%%%%%%%%%%%%%%%%%%%%%%%%
\subsection{Binary interactions dynamics}
We consider binary interaction of opinions, both for the dynamics of the followers and the leaders, in presence of additional random processes. This modelling choice aims to include exogenous factors which can not be described by the constrained deterministic process.

More precisely, the post-interaction opinions of two leaders within the group $k$, are obtained as follows 
\be\begin{split}\label{eq:binary_leaderp}
v^{\prime} &= v + \alpha S^k(v,v_*)(v_*-v)+2\alpha \sum_{\ell=1}^M u^\ell + \eta^k D^k_L(v),\\
{v_*}^{\prime} &= v_* + \alpha S^k(v_*,v)(v-v_*)+2\alpha \sum_{\ell=1}^Mu^\ell + \eta^k_* D^k_L(v_*),
\end{split}\ee
where the best reply strategy $ \sum_{\ell=1}^M u^\ell$ is the same for every population of leaders, and under assumption \eqref{suffcondm}, is defined for every $\ell=1,\ldots,M$ by \eqref{eq:bin_ctrlB_sol_avem}.
%\begin{align}\label{eq:bin_ctrlB_sol_avem} 
%F^\ell_\alpha   = \frac{\beta^\ell}{2\alpha}\bar B^\alpha_{\ell} \left(\psi^{\ell}\vl+\mu^{\ell}\mF-m_L^{\ell}\right), 
%\end{align}
%equivalently to \eqref{eq:bin_ctrlB_sol_avem}, where the coefficients $\bar B^\alpha_{\ell}$ are defined through the inverse of \eqref{syscontrolm}, ${\bf B} ^{\alpha}=[B^\alpha_{k\ell}]$,  as in \eqref{eq:sumB}, and $\mF,m_L^k$ are the average opinions defined in \eqref{eq:mF}.
The random variables $\eta^k,\eta^k_* $, with law $\Theta_{\eta^k}(\cdot)$, have zero mean and standard deviation $\sigma_{\eta^k}$, 
%and finite moment up to the third order 
whereas the function $0\le D^k_L(\cdot)\le 1$ represents the local relevance of diffusion of the leaders' population. 

In order to describe the  interaction among followers as a binary exchange of informations we split the interaction among follower-follower, induced by the kernel $P(\cdot,\cdot)$, and follower-leader interaction  induced by $R^k(\cdot,\cdot)$, for each leaders' population $k =1,\ldots,M$.

Thus, the interactions between followers are given by 
\be\begin{split}\label{eq:binary_follower_opinion}
&w^{\prime} = w+\alpha P(w,w_*)(w_*-w)+\xi D_F(w),\\
&w_*^{\prime} = w_*+\alpha P(w_*,w)(w-w_*)+\xi_* D_F(w_*),
\end{split}\ee
where $(w,w_*)\in I$ are the pre-interaction opinions and $w^{\prime},w_*^{\prime}$ the opinions after the exchange of informations between the two agents. Moreover, $0\le D_{F}(\cdot)\le 1$ is a local diffusion function which scales the strength of the random variables $\xi,\xi_*$ in the binary dynamics, whereas $\xi,\xi_*$ have law $\Theta_\xi(\cdot)$ with zero mean and standard deviation $\sigma_\xi$.
% and with at least finite third order moment.

Finally, each interaction between leaders and followers is given by the binary dynamics
\be\begin{split}\label{eq:binary_follower_leader}
&w^{\prime\prime} = w+\alpha R^k(w,v)(v-w)+\xi^k D^k_{FL}(w), \\
&{v}^{\prime\prime}=v,
\end{split}\ee
where the local diffusion functions $0\le D^k_{FL}(\cdot)\le 1$ modulate the impact of the random variables $\xi^k$ on the binary dynamics. The random variables $\xi^k$  have law $\Theta_{\xi^k}(\cdot)$ and standard deviations $\sigma_{\xi^k}$.
%, taking values in a Borel set $\mathcal B_{\xi^k}\subseteq \RR$ and with at least finite third order moment.

%\subsubsection{Bounds performance}
In the following we give sufficient conditions to preserve the opinion bounds in the aforementioned  dynamics.
\begin{proposition}
Let us consider the binary dynamics described by \eqref{eq:binary_leaderp}-\eqref{eq:binary_follower_leader}. 
\begin{enumerate}
\item 
For each $k=1,\dots,M$ let us define the quantities $c^k_{\pm}$, $d_{\pm}^k$ as follows
\[
c_{\pm}^k = \min_{v\in [-1,1]}\left\{ \dfrac{\pm 3\sum_{\ell=1}^M \beta^\ell \bar{B}^{\alpha}_{\ell}}{D^k_L(v)}, D^k_L(v)\ne 0\right\},\qquad
d_{\pm}^k = \min_{v\in [-1,1]}\left\{ \dfrac{1\mp v}{D_L^k(v)} \right\}.
\]
Therefore, if $v,v_*\in I$ then $v^\prime, v_*^\prime\in I$ provided 
\[
d_-^k - c^k_-\le \eta^k \le d_+^k -c^k_+.
\]
\item Let $F_{\pm}$ be defined as follows
\[
F_{\pm} = \min_{w\in[-1,1]}\left\{\dfrac{1\mp w}{D_{F}(w)}, D_{F}(w)\ne 0 \right\}.
\]
Therefore, if $w,w_*\in I$ then $w^{\prime},w'_*\in I$ provided
\[
(1-\alpha)F_-\le \xi \le (1-\alpha)F_+.  
\]
\item Let $L_{\pm}^k$ be defined for all $k=1,\dots,M$ as follows
\[
L^k_{\pm} = \min_{w\in[-1,1]}\left\{\dfrac{1\mp w}{D_{FL}^k(w)}, D_{FL}^k(w)\ne 0 \right\}.
\]
Therefore, if $w\in I$ then $w^{\prime\prime}$ provided
\[
(1-\alpha)L^k_-\le \xi^k \le (1-\alpha)L^k_+.  
\]
\end{enumerate}
\end{proposition}
\begin{proof}
Since all compromise functions lie in the interval $[0,1]$, the proof is a straightforward consequence of the results in \cite{AHP,APTZ,APZa,PT2,T} and of the fact that $\mathbf B^{\alpha}$ is strictly diagonally dominant. We omit the details.
\end{proof}
%From the definition of binary interaction between followers in \eqref{eq:binary_follower_opinion}, in absence of diffusion, the boundaries are never violated. Indeed since $|w|\le 1$ it follows that $|w-w_*|\le 1$ and, since $0\le P(\cdot,\cdot)\le 1$, it is easily seen that $w^\prime$, $w_*^\prime$ lies in the interval $[-1,1]$. At last, the post--interaction opinion of the followers, in the leader--follower dynamics takes values in the interval $[-1,1]$ if the following proposition yields
%\begin{proposition}
%Let $L_{\pm}^k$ be defined for all $k=1,\dots,M$ as follows
%\[
%L^k_{\pm} = \min_{w\in[-1,1]}\left\{\dfrac{1\mp w}{D_{FL}^k(w)}, D_{FL}^k(w)\ne 0 \right\},
%\]
%Therefore, if $w\in I$ then $w^{\prime\prime}$ provided
%\[
%(1-\alpha)L^k_-\le \xi^k \le (1-\alpha)L^k_+.  
%\]
%\end{proposition}

\subsection{The Boltzmann--type description}
The evolution in time of followers'  and leaders' density functions, $f(w,t), g^k(v,t)$, is given by the following system of integro--differential equations of the Boltzmann-type 
\be\begin{split}\label{eq:system_Q}
&\dfrac{\partial }{\partial t}f(w,t) = Q_F(f,f)(w,t)+\sum_{\ell=1}^M Q_{FL}^{\ell}(f,g^\ell)(w,t),\\
&\dfrac{\partial}{\partial t} g^k(v,t) = Q_L^k(g^k,g^k)(v,t)% \textcolor{red}{+\sum_{\ell\neq k}^M Q_{LL}^{k\ell}(g^k,g^\ell)(v,t)} 
,\qquad k = 1,\dots,M
\end{split}\ee
with $f(w,0)=f_0(w)$ and $g^k(v,0)=g_0^k(v)$. In \eqref{eq:system_Q} the collision operators $Q_{F}(\cdot,\cdot), Q_{FL}^{k}(\cdot,\cdot)$, $Q_L^k(\cdot,\cdot)$, and $Q_{LL}^{k\ell}(\cdot,\cdot)$
are defined as follows
\be\begin{split}
&Q_F(f,f) (w,t) =\\
&\quad \int_{\mathcal B^2}\int_{I} \left(^{\prime}B_F\dfrac{1}{J_F}f(^{\prime}w,t)f(^{\prime}w_*,t)-B_F f(w,t)f(w_*,t)\right)dw_*d\xi d\xi_*,
\end{split}\ee
where we indicated with $(^{\prime}w,^{\prime}w_*)$ the pre-interaction opinions given by $(w,w_*)$ after the interaction. The term $J_F=J_F(w;w_*)$ denotes as usual the Jacobian of the transformations $(w,w_*)\rightarrow (w^{\prime},w_*^{\prime})$. The kernels $^{\prime}B_F,B_F$ characterize the binary interaction and in following will be considered of the form
\be\label{eq:BF}
B_F= c_F\; \Theta_\xi(\xi)\Theta_\xi(\xi_*)\chi(|w^{\prime}|\le 1)\chi(|w_*^{\prime}|\le 1),
\ee
where $c_F>0$ is a scaling constant indicating the interaction frequency and $\chi(\cdot)$ is the indicator function. Similarly, we define for $k = 1,\dots,M$
\be\begin{split}\label{QFL}
&Q_{FL}^k(f,g^k)(w,t) =\\
&\quad \int_{\mathcal B^2}\int_I \Big({^\prime}  B_{FL}^k\dfrac{1}{ J_{FL}^k}f(^{\prime}w,t)g^k(^{\prime}v_*,t) -B_{FL}^k f(w,t)g^k(v_*,t)\Big)dv_*d\xi d\xi_*,
\end{split}\ee
where we indicate again with $J_{FL}^k$ the Jacobian term of the transformation $(w,v_*)\rightarrow (w^\prime,v_*^\prime)$ and the kernel $B_{FL}^k$ has the form 
\be
B_{FL}^k = c_{FL}^{k}\; \Theta_{\xi^k}(\xi^k)\Theta_{\xi^k}(\xi_*^k)\chi(|w^\prime|\le 1)\chi(|v_*^\prime|\le 1),
\ee
with $c_{FL}^k>0$ a scaling parameter. Finally, for each the leaders' population we define the operator $Q_L^k(\cdot,\cdot)$ as
\be\begin{split}
&Q^k_L(g^k,g^k)(v,t) =\\
&\quad \int_{\mathcal B^2}\int_{I} {^\prime B_L^k} \dfrac{1}{J_L^k}g^k(^{\prime}v,t)g^k({^{\prime}{v}}_*,t)-B_L^k g^k(v,t)g^k(v_*,t)dv_* d\eta d\eta_*,  
\end{split}\ee
with 
\be
B_L^k =c_L^k\Theta_{\eta^k}(\eta)\Theta_{\eta^k} (\eta_*)\chi(|v|\le 1)\chi(|v_*|\le 1),
\ee
and $c_L^k>0$ is the interaction frequency of the leaders of the $k$th population.

\subsection{Evolution of the mean opinion and variance}
In order to have an insight on the global behaviour of model \eqref{eq:system_Q} we study here, under simplified assumptions, the evolution of the momentum of the the followers' density $f(w,t)$ and leaders' densities $g^k(v,t)$.
Hence, we introduce the weak formulation of the system \eqref{eq:system_Q} for a given test function $\phi(\cdot,t)\in\mathcal{C}^2(I)$  at $t\ge 0$ as follows
\be\begin{aligned}\label{eq:WBOL_F}
&\dfrac{d}{d t}\int_{ I}\phi(w)f(w,t)dw =\\
&\quad \dfrac{c_F}{2} \Big< \int_{I^2}(\phi(w^\prime)+\phi(w_*^\prime)-\phi(w)-\phi(w_*)) f(w,t)f(w_*,t)dwdw_*\Big>\\
&\quad+\sum_{k=1}^M c_{FL}^k\Big<\int_{ I^2}(\phi(w^{\prime\prime})-\phi(w)) f(w,t)g^k(v,t)dwdv \Big>
\end{aligned}\ee
and for all $k=1,\dots,M$
\be\begin{split}\label{eq:WBOL_L}
&\dfrac{d}{d t}\int_{I}\phi(v) g^k(v,t)dv =\\
&\quad \dfrac{c_L^k}{2} \Big< \int_{I^2}(\phi({v}^\prime)+\phi({v_*}^\prime)-\phi(v)-\phi(v_*))
g^k(v,t)g^k(v_*,t)dvdv_*\Big>,
\end{split}\ee
for the initial densities $f_0(w),g_{0}^k(v)$, and where $\langle\cdot\rangle$ denotes the expectation with respect to the random variable. 
From the weak formulation of the Boltzmann model \eqref{eq:WBOL_F}-\eqref{eq:WBOL_L} we can derive the evolution of macroscopic quantities for the opinion by choosing as a test function $\phi(w)=1,w,w^2$. 

Taking $\phi(w)=1$ we obtain immediately the conservation of the total number of leaders and followers in the system \cite{PT2,T}. Whereas for $\phi(w) = w$ we recover the evolution of $m_F(\cdot)$ and $m_L^k(\cdot)$, defined in \eqref{eq:mF}, as solutions of the following system of $M+1$ differential equations
\[\begin{split}\label{eq:ME}
\dfrac{d}{dt} m_F(t) &= \dfrac{c_F\alpha}{2}\int_{I^2}[P(w,w_*)-P(w_*,w)](w_*-w)f(w,t)f(w_*,t)dw\; dw_*\\
&+\sum_{\ell=1}^M  c_{FL}^{\ell} \alpha \int_{I^2}R^{\ell}(w,v)(v-w)f(w,t)g^{\ell}(v,t)dvdw,\\
\dfrac{d}{dt} m_L^k(t) &= \dfrac{c_L^k\alpha}{2\rho^k}\int_{I^2}\Big\{[S^k(v,v_*)-S^k(v_*,v)](v_*-v)\\
& +4\sum_{\ell=1}^M u^{\ell}\Big\}g^k(v,t)g^k(v_*,t)dv\;dv_*.
\end{split}\]
Under the simplifying assumptions $P(\cdot,\cdot)$ and $S^k(\cdot,\cdot)$, for all $k=1,\dots,M$, symmetric, in the case $R^{\ell}\equiv 1$ for all $\ell=1,\dots,M$ and for a control term as in \eqref{eq:bin_avesol}, we reduce to the following differential system
\be\begin{split}\label{eq:system_eqdiff}
\dfrac{d}{dt}m_F(t) &= \sum_{\ell=1}^M c_{FL}^{\ell}\rho^{\ell} \alpha \left( m_L^{\ell}(t)-m_F(t) \right),\\
\dfrac{d}{dt}m_L^k(t) &= \dfrac{\beta}{1+(M-1)\beta}\left(c_L^k \rho^k \sum_{\ell=1}^M(\psi^\ell\vl+\mu^\ell m_F-  m_L^\ell)\right).
\end{split}\ee
If we further assume $\rho^k=\rho$, $c_L^k=c_{FL}^k=1/M$ for all $k=1,\dots,M$ the evolution of the average leaders' opinion $$\bar m_L(t)=\frac1{M}\sum_{k=1}^M m_L^k(t),$$ is explicitly computable from
\be\begin{split}\label{eq:mean_reduced}
\dfrac{d}{dt} m_F(t) &= \rho\alpha \left( \bar m_L(t)-m_F(t) \right)\\
\dfrac{d}{dt} \bar m_L(t)   &= \dfrac{\beta \rho}{1+(M-1)\beta} \left[ \frac1{M}\sum_{\ell=1}^M\left( \psi^\ell \vl +\mu^\ell m_F(t)\right)-\bar m_L(t)\right].
\end{split}\ee
Therefore, since $\mu^\ell=1-\psi^\ell$ both $m_F$ and $\bar m_L$ converge exponentially towards the asymptotic state 
\be
\bar v=\frac{\sum_{\ell=1}^M \psi^\ell \vl}{\sum_{\ell=1}^M \psi^\ell}.
\label{eq:asm}
\ee
As a result the average opinion of the followers is a weighted average of the different desired states of the leaders' groups.

The evolution of the second order moments $E_F(t)$, $E_L(t)$ can be obtained from \eqref{eq:WBOL_F}--\eqref{eq:WBOL_L} with $\phi(w)=w^2$, $\phi(v)=v^2$
\be\label{eq:EN}
\begin{split}
&\dfrac{d}{dt} E_F(t) =\\
&\quad \dfrac{c_F}{2} \Big<\int_{I^2} \Big(( w+\alpha P(w,w_*)(w_*-w)+\xi D_F(w))^2 +( w_*+\alpha P(w_*,w)(w-w_*)\\
&\quad +\xi D_F(w_*))^2-w^2-w_*^2\Big)f(w,t)f(w_*,t)dwdw_*\Big>\\
&\quad +\sum_{\ell=1}^M c_{FL}^\ell \Big< \int_{I^2}((w+\alpha R^\ell(w,v)(v-w)+\xi^k D_{FL}^\ell(w))^2-w^2)\\
&\qquad f(w,t)g^\ell(v,t)dwdv \Big>, \\
&\dfrac{d}{dt}E_L^k(t) = \dfrac{c_L^k}{2\rho^k} \Big< \int_{I^2} ((v+\alpha S^k(v,v_*)(v_*-v)+2\alpha \sum_{\ell=1}^M u^\ell+\eta^k D_L^k(v))^2\\
&\qquad +(v_*+\alpha S^k(v_*,v)(v-v_*)+2\alpha \sum_{\ell=1}^M u^\ell+\eta^k D_L^k(v_*))^2-v^2-v_*^2) \Big> \\
&\qquad g^k(v,t)g^k(v_*,t)dvdv_*.
\end{split}
\ee
Equation \eqref{eq:EN} together with \eqref{eq:ME} form a closed system for the evaluation of the second order moment of the asymptotic opinion of leaders and followers.

\section{An asymptotic Fokker--Planck game}\label{sect:FP}
In order to study the long-time behaviour of  the system \eqref{eq:system_Q} we introduce here the so-called {\em quasi-invariant scaling},  which allows to pass from a system of Boltzmann equations to a system of Fokker-Planck-type equations. This approach for socio-economic problems was firstly introduced in \cite{CPT05}, and subsequently developed for control problems in \cite{AHP, APZa}, and shares some similarities with the so called grazing collision limit of plasma physics, we refer to \cite{APTZ, PT2} for further discussions. 

\subsection{Fokker-Planck scaling}
Following \cite{APZa} we introduce the scaling parameter $\varepsilon>0$, and rescale as follows the interaction frequencies 
\begin{subequations}\label{eq:scaling}
\be\label{eq:scaling_1}
c_{F} \rightarrow \dfrac{1}{\varepsilon },\qquad c_{FL}^k\rightarrow \dfrac{c_{FL}^k}{\varepsilon\rho^k}, \qquad c^k_{L} \rightarrow \dfrac{1}{\varepsilon\rho^k},
%{\color{red}\qquad c_{LL}^{k\ell}\to \frac{c_{LL}^{k\ell}}{\varepsilon\rho^\ell}}
\ee
the binary parameters 
\be\begin{split}\label{eq:scaling_2}
\alpha = \varepsilon, \qquad \beta^k=\dfrac{4\varepsilon^2}{\nu^k+4\varepsilon^2},\qquad  \nu^k\rightarrow \varepsilon\nu^k,\end{split}\ee
and the standard deviations for the noise
\be\begin{split}\label{eq:scaling_3}
\sigma_{\xi}\rightarrow \sqrt{\varepsilon}\sigma_{\xi},\qquad \sigma_{\xi^k}\rightarrow \sqrt{\varepsilon}\sigma_{\xi^k},\qquad\sigma_{\eta^k}\rightarrow \sqrt{\varepsilon}\sigma_{\eta^k}.
%{\color{red}\qquad \sigma_{{\eta}^{k\ell}}\to \sqrt{\varepsilon}\sigma_{{\eta}^{k\ell}}}
\end{split}\ee
\end{subequations}
Introducing the above scaling in the Boltzmann system \eqref{eq:WBOL_F}-\eqref{eq:WBOL_L}, and taking the quasi-invariant limit $\varepsilon\rightarrow 0$ we have the following result for the evolution of the density functions $f(w,t)$ and $g^k(v,t)$.
\begin{theorem}
Let us consider the Boltzmann system \eqref{eq:WBOL_F}-\eqref{eq:WBOL_L} under the quasi--invariant opinion scaling \eqref{eq:scaling}. Let the initial densities $f_0=f_0(w)$, $g_0^k=g_0^k(w)$, $k=1,\dots,M$ be probability measures. In the limit $\varepsilon \to 0$, the weak solutions of the Boltzmann system $f_{\varepsilon}(w,t)$, $g^k_{\varepsilon}(v,t)$, converge,  up to extraction of a subsequence, to the probability densities $f(w,t)$, $g^k(v,t)$. These densities are weak solution of the Fokker--Planck system
\be
\begin{split}\label{eq:FPfollower}
&\partial_t f(w,t) + \partial_w\left( \left(\mathcal P[f](w,t) +\sum_{\ell=1}^M c^\ell_{FL}\mathcal R^\ell[g^\ell](w,t)\right) f(w,t)\right) =\\
& \qquad\qquad\qquad\dfrac{\sigma_{\xi}^2}{2}\partial^2_w (D_F^2(w)f(w,t)) + \sum_{\ell=1}^M \dfrac{\sigma_{\xi^\ell}^2}{2}\partial^2_w ((D_{FL}^\ell)^2(w)f(w,t)) %\quad (\textrm{Se  } \sigma_{\xi}\equiv\sigma_\xi^k)
\end{split}
\ee
with
\be\begin{split}\label{eq:PRk_def}
\mathcal P[f](x,w,t) &= \int_{I}P(w,w_*)(w_*-w)f(w_*,t)dw_*,\\
\mathcal R^k[g^k](w,t)&=\dfrac{1}{\rho^k}\int_{ I}R^k(w,v)(v-w)g^k(v,t)dv,\qquad k=1,\ldots,M,
\end{split}\ee
and for each $k=1,\ldots,M$ 
\be
\begin{split}\label{eq:FPleader}
&\partial_t g^k(v,t)+\partial_v\left(\left( \mathcal S^k[g^k](v,t)+ \sum_{\ell=1}^M\mathcal  U[f,g^\ell](t)\right) g^k(v,t)\right)= \\
&\qquad\qquad\qquad\qquad\qquad\qquad\qquad\qquad\qquad\dfrac{\sigma_{\eta^k}^2}{2}\partial ^2_v \Big( (D^k_L(v))^2g^k(v,t) \Big),
\end{split}
\ee
with 
\be\begin{split}
\mathcal U [f,g^\ell](t) &= \frac{2}{\nu^\ell}\Big(\psi^\ell\vl+ \int_I\int_I ( \mu^\ell wf(w,t)-\dfrac{1}{\rho^\ell}vg^\ell(v,t))dvdw\Big) % \mF(t)-m^\ell_L(t) \right) ,\qquad \ell=1,\ldots,M.
\end{split}\ee
\end{theorem}
\proof
Note that in the Fokker-Planck model for the leaders the terms $\mathcal U[f,g^\ell](t)$ are obtained from \eqref{eq:bin_ctrlB_sol_avem} in the quasi-invariant limit. In fact, from \eqref{eq:scaling}, and for $\varepsilon\to 0$ we have 
 \begin{align}
 \beta^k = \frac{4\varepsilon}{\nu^k+4\varepsilon} \to 0,\qquad \frac{\beta^k}{2\alpha} = \frac{2}{\nu^k+4\varepsilon} \to \frac{2}{\nu^k}.
 \end{align}
Thus thanks to the last relation we can show that the matrix of system \eqref{syscontrolm}, where the various strategies $F^\alpha_\ell$ are computed, reduce to the identity matrix, since $\beta^k\to0$ for every $k=1,\ldots,M$ and therefore $\bar B^\alpha_\ell = 1$  for every $\ell=1,\ldots,M$.
The rest of the proof is based on standard arguments, and we refer to \cite{T,PT2} for more details.
\endproof

\subsection{Asymptotic states}
Thanks to its simpler structure, under some simplification assumptions, it is possible to compute explicitely stationary states of the Fokker-Planck system \eqref{eq:FPfollower}--\eqref{eq:FPleader}.
We look for steady states $\fs(w),\gs^k(w)$ solutions to the following system of equations
\be
\begin{split}\label{stationary}
&\partial_w\left(\left(\sigma_{\xi}^2 D_F^2(w) + \sum_{\ell=1}^M {\sigma_{\xi^\ell}^2} (D_{FL}^\ell)^2(w)\right)\fs(w)\right)=\\
&\qquad\qquad\qquad\qquad2\left(\mathcal P[\fs](w) +\sum_{\ell=1}^M c^\ell_{FL}\mathcal R^\ell[\gs^\ell](w)\right)\fs(w),
\\
& \partial_v \left(\sigma_{\eta^k}^2(D^k_L(v))^2\gs^k(v)\right)
=\\
&\qquad\qquad2\left(\mathcal S^k[\gs^k](v)+\sum_{\ell=1}^M\frac{2}{\nu^\ell}\left(\psi^\ell\vl+\mu^\ell m_{F,\infty}-m^\ell_{L,\infty} \right)\right)\gs^k(v).
\end{split}
\ee
In order to solve the above set of equations we assume all interaction functions to be unitary constants, i.e $P = S^k = R^k = 1$, for $k=1,\ldots,M$ and that the local diffusion functions have the form \cite{T}
\be
D_F(w)=D_{FL}^k(w) =D(w)=(1-w^2),\qquad D_L(v) = (1-v^2).
\label{eq:diff}
\ee
Thus, the previous system reads
\be
\begin{split}\label{eq:stationary}
&\left((m_{F,\infty}-w)+\sum_{\ell=1}^M c^\ell_{FL}(m_{L,\infty}^\ell-w)\right)\fs(w)=\cr
&\qquad\qquad\qquad\qquad\qquad\qquad\dfrac{\left(\sigma_{\xi}^2+ \sum_{\ell=1}^M {\sigma_{\xi^\ell}^2}\right)}{2}\partial_w\left(D(w)^2\fs(w)\right)
\\
&\left((m_{L,\infty}^k-v)+\sum_{\ell=1}^M\frac{2}{\nu^\ell}\left(\psi^\ell\vl+\mu^\ell m_{F,\infty}-m^\ell_{L,\infty} \right)\right)\gs^k(v)=\cr
&\qquad\qquad\qquad\qquad\qquad\qquad\dfrac{\sigma_{\eta^k}^2}{2} \partial_v \left(D_L^k(v)^2\gs^k(v)\right),
\end{split}
\ee
for all $k=1,\dots,M$. Finally, under the further assumptions $\nu^k=\nu=2M$, $ c_L^k=c_{FL}^k=1/M$ from \eqref{eq:mean_reduced} we have that $\bar m_{L,\infty}=m_{F,\infty}=\bar v$, with $\bar v$ given by \eqref{eq:asm}. Therefore, the system reduces to
\be
\begin{split}\label{eq:stationary2}
&\left(\bar v - w\right) \fs(w)=\dfrac{\left(\sigma_{\xi}^2+ \sum_{\ell=1}^M {\sigma_{\xi^\ell}^2}\right)}{4}\partial_w\left(D(w)^2\fs(w)\right)
\\
&\left[(m_{L,\infty}^k-v)+\frac1{M}\sum_{\ell=1}^M\psi^\ell(\vl-\bar v)\right]\gs^k(v)=\dfrac{\sigma_{\eta^k}^2}{2} \partial_v \left(D_L^k(v)^2\gs^k(v)\right),
\end{split}
\ee
which can be solved to give for $k=1,\dots,M$ 
\be
\begin{split}
f_{\infty}(w) &= \gamma_F (1+w)^{-2+\bar v/(2\sigma^2_F)}(1+w)^{-2-\bar v/(2\sigma^2_F)}
\exp\left\{ -\dfrac{1-\bar v w}{\sigma^2_F(1-w^2)} \right\},\\
g_{\infty}^k(v)& = \gamma^k_L (1+w)^{-2+b^k_L/(2\sigma_{\eta^k}^2)}(1+w)^{-2-b^k_L/(2\sigma_{\eta^k}^2)}
\exp\left\{ -\dfrac{1-w\; b^k_L}{\sigma_{\eta^k}^2(1-w^2)} \right\},
\end{split}
\label{eq:statio}
\ee
with $\sigma^2_F= \dfrac{\sigma_{\xi}^2+\sum_{\ell=1}^M \sigma^2_{\xi^\ell}}{2}$, 
$b_L =m_{L,\infty}^k+\dfrac1{M}\sum_{\ell=1}^M\psi^\ell(\vl-\bar v)$ and $\gamma_F,\gamma^k_L$ normalization constants. 

\begin{rmk}~
\begin{itemize}
\item The stationary solutions for the leaders in \eqref{eq:statio} are defined implicitly from the average asymptotic opinion $m_{L,\infty}^k$ which depends on the initial average opinions $m_L^k(0)$ and can be computed from \eqref{eq:mean_reduced}.   
\item Other choices of the local diffusion functions $D_F(w)$, $D_{FL}^k(w)$ and $D(w)$ originate different stationary solutions, we refer to \cite{T, PT2} for further details.
\end{itemize} 
\end{rmk}

\section{Boltzmann games with heterogeneous followers}\label{sect:FP_het}
Clearly, a more realistic model should take into account the presence of additional effects characterized by the heterogeneous nature of the followers' population. In particular, experimental literature in decision science focussed on the impact of knowledge in communication effectiveness and in group decision--making tasks being the attitude of the audience towards the communicators dependent on their perceived high/low credibility level, see \cite{BOLRRF,HW}. 

\subsection{Modelling heterogeneous knowledge}
To describe the heterogeneity of the followers' knowledge we proceed similarly to \cite{BT,PT1,PVZ} by introducing an additional scalar variable $x\in \RR_+$ which measures the level of knowledge. At each interaction an agent achieves a certain level of knowledge from a background, representing the social environment, given in terms of the random variable $z\in\RR_+$ with distribution $C(z)$ with bounded mean 
\[
\int_{\RR_+}  C(z)dz = 1, \qquad \int_{\RR_+}zC(z)dz = m_B.
\]
Moreover, when the agents interact they also exchange part of their knowledge with other agents. The resulting dynamics, at a binary level, is the following   
\be\begin{split}\label{eq:binary_follower_knowledge}
&x^{\prime} = (1-\alpha\lambda(x))x +\alpha \lambda_C(x)x_*+\alpha\lambda_B(x)z+\kappa x\\
&x_*^{\prime} = (1-\alpha\lambda(x_*))x_*+\alpha\lambda_C(x_*)x+\alpha\lambda_B(x_*)z+\kappa_* x_*,
\end{split}\ee
where $\kappa,\kappa_*$ are random variables with zero mean and finite second order moment $\sigma_{\kappa}^2$ taking into account the unpredictable changes of the process. We can easily prove the following result \cite{PT1,PVZ}
\begin{proposition}
Let $\lambda \in[\lambda_-,\lambda_+]$ with $\lambda_->0$, $\lambda_+<1$ and $\lambda_C(x),\lambda_B(x)\in[0,1]$ for all $x\in \RR^+$ then the post--interaction knowledge $x^\prime$ is still positive if $\kappa\ge -1+\lambda_+$. 
\end{proposition}

In order to introduce the action of the knowledge's heterogeneity in the opinion  dynamics we consider a compromise function between followers which depends on both the agents' opinion and knowledge. A possible structure for this function is given by
\be\label{eq:P_def}
P(w,w_*;x,x_*) = H(w,w_*)K(x,x_*),
\ee
where $0\le H(\cdot,\cdot)\le 1$ is a positive compromise propensity depending only by the opinion variable and $0\le K(\cdot,\cdot)\le 1$ is a function taking into account the knowledge of the two interacting agents \cite{APTZ,APZc}. A natural choice for this function is the following (see \cite{PVZ})
\be\label{eq:def_cMC}
K(x,x_*) = \dfrac{1}{1+e^{a(x-x_*)}}, \qquad a>1,
\ee
modelling the interaction propensity in terms of the knowledge gap $x-x_*$ between the two interacting individuals. 

Furthermore, to model the leaders' credibility, we introduce an interaction function in the leader--follower dynamics
 $0\le R\le 1$ with the following form
\be\begin{split}\label{eq:def_R}
R^k(w,v;x)&=H(w,v)K(x,\Psi(|v-m_L^k(0)|)),\\
\Psi(|v-m_L^k(0)|)& = \dfrac{1}{(\varsigma +|v-m_L^k(0)|)^{\gamma}}, \quad \varsigma,\gamma>0,
\end{split}\ee
where $m_L^k(0)$ is the initial mean opinion of the $k$th population of leaders, see \eqref{eq:mF}, the function $\Psi:[0,2]\rightarrow \RR^+$ is a credibility index measuring the distance at a given time of the leader's opinion with respect to its initial position and the function $K(\cdot,\cdot)$ acts as described in \eqref{eq:def_cMC}. 

\subsection{Kinetic games with heterogeneous effects}
In order to introduce the Boltzmann--type system of equations defining the leaders' game with heterogeneous followers' dynamics we consider the distribution $f(x,w,t)$ of individuals with knowledge $x\in X\subseteq\RR_+$ and opinion $w\in[-1,1]$ at time $t\ge 0$. The evolution is then given by the following system of Boltzmann equations
\be\begin{split}\label{eq:system_Qk}
&\dfrac{\partial }{\partial t}f(x,w,t) = Q_F(f,f)(x,w,t)+\sum_{\ell=1}^M Q_{FL}^{\ell}(f,g^\ell)(x,w,t),\\
&\dfrac{\partial}{\partial t} g^k(v,t) = Q_L^k(g^k,g^k)(v,t)
,\qquad k = 1,\dots,M
\end{split}\ee
where the operators $Q_F(\cdot,\cdot)$, $Q_{FL}^\ell(\cdot,\cdot)$ are defined as follows
\[\begin{split}
Q_F(f,f)(x,w,t) = &\int_{\RR_+}\int_{\mathcal B^2}\int_{X}\int_{I} C(z)\Bigg({}^\prime B_F\dfrac{1}{J_F}f({}^\prime x,{}^\prime w,t)f({}^\prime x_*,{}^\prime w_*,t)\\
&\qquad\qquad-B_F f(x,w,t)f(x_*,w_*,t)\Bigg)dx_*\,dw_*\,d\xi\, d\xi_*\,dz \\
Q_{FL}^{\ell}(f,f)(x,w,t) = &\int_{\RR_+}\int_{\mathcal B^2}\int_{X}\int_{I} C(z)\Bigg({}^\prime B_{FL}^\ell\dfrac{1}{J_{FL}^\ell}f({}^\prime x,{}^\prime w,t)g^\ell({}^\prime v_*,t)\\
&\qquad\qquad-B_{FL}^\ell f(x,w,t)g^\ell(v_*,t)\Bigg)dv_*\,d\xi\, d\xi_*\, dz,
\end{split}\]
being now
$$B_F = c_F \Theta_{\xi}(\xi)\Theta_{\xi_*}(\xi_*)\chi(|w^\prime|\le 1)\chi(|w_*^\prime|\le 1)\chi(x^\prime\ge 0)\chi(x_*^\prime\ge 0)$$
 and 
$$B_{LF}^\ell = c_{FL}^\ell \Theta_{\xi^\ell}(\xi^\ell)\Theta_{\xi_*^\ell}(\xi_*^\ell)\chi(|w^\prime|\le 1)\chi(x^\prime\ge 0)\chi(|v_*^\prime|\le 1).$$
%for given initial distributions $f(x,w,0)=f_0(x,w)$, $g^\ell(v,0)=g^\ell_0(v)$. 

Similarly, under the quasi invariant scaling \eqref{eq:scaling}-\eqref{eq:scaling_3}, if we further rescale $\sigma_{\kappa}\rightarrow \epsilon\sigma_{\kappa}$ we obtain the following system of Fokker--Planck equations 
\be
\begin{split}\label{FPfollower}
&\partial_t f(x,w,t)+\Big[\partial_x \mathcal C[f](x,w,t)+ \partial_w \Big(\mathcal P[f](x,w,t)\\
&\quad +\sum_{\ell=1}^M c^\ell_{FL}\mathcal R^\ell[g^\ell](x,w,t)\Big) f(x,w,t)\Big] = \dfrac{\sigma_{\kappa}^2}{2}\partial_x^2(x^2f(x,w,t))\\
&\quad +\dfrac{\sigma_{\xi}^2}{2}\partial^2_w (D_F^2(w)f(x,w,t)) + \sum_{\ell=1}^M \dfrac{\sigma_{\xi^\ell}^2}{2}\partial^2_w ((D_{FL}^\ell)^2(w)f(x,w,t)) 
\end{split}
\ee
where $\mathcal P[\cdot]$ and $\mathcal R^k[\cdot]$ have been defined in \eqref{eq:PRk_def} and $\mathcal C[\cdot]$ is given by
\[
\mathcal C[f](x,w,t) = \int_{\RR^+}\int_{\RR^+\times I}(-\lambda x+\lambda_Cx_*+\lambda_Bz)f(x_*,w_*,t)C(z)dx_*dw_*dz. 
\]

\section{Numerical examples and applications}\label{sect:num}
We propose in the present section several numerical examples for the Boltzmann--type model \eqref{eq:system_Q} and \eqref{eq:system_Qk} in the quasi-invariant scaling \eqref{eq:scaling_1}-\eqref{eq:scaling_3} with $\epsilon=0.01$. All the results have been obtained though a direct Monte Carlo simulation of the dynamics, see \cite{PR,PT2} for a description of the methods. In all the numerical tests we assumed that $10\%$ of the population is composed by leaders, equally divided for each family. For clarity in all figures the leaders' profiles have been magnified by a factor $10$. The regularization terms of the controls have been fixed to $\nu^p=0.1$. 

We considered as local diffusion function $D(x,w)=1-w^2$, hence the diffusion does not act on the agents with extreme opinions. The random variables $\eta,\eta^A,\eta^B$ are uniformly distributed with scaled variances $\sigma_{\eta}^2$, $\sigma_{\eta^A}^2$, $\sigma_{\eta^B}^2$. The knowledge dynamics is characterized by $\lambda(x)=\lambda>0$, $\lambda_C(x)=\lambda_C>0$ and $\lambda_B(x)=\lambda_B$. The random variable $\kappa$ and $z$ are uniformly distributed. 
In all test cases we assume $P(\cdot,\cdot;\cdot,\cdot)$ of the form \eqref{eq:P_def} and $R(\cdot,\cdot;\cdot,\cdot)$ of the form \eqref{eq:def_R}. 

\begin{table}[t]
\centering
\caption{Computational parameters for the different test cases.}
\begin{tabular}{c|c|c|cccc|c|cc|ccc}
\hline
\hline
Test & $L_k$ & $S^k(\cdot)$ &$\vk$& $\psi^k$&$\nu^k$ &$\sigma_{\eta^k}$ &$F$&$P(\cdot)$& $\sigma_{\eta}$ &$R^k(\cdot)$&$c_{FL}^k$&$\sigma_{\eta^k}$\\
\hline
\hline
 I & $1$&1&    0.5& 0.5&0.1 &0.01& &\eqref{PRT1}&0.01&\eqref{PRT1}&0.1&0.01\\
    & $2$&1& -0.5& 0.5&0.1 &0.01&  &                     &      &\eqref{PRT1}&0.1&0.01\\
\hline
\hline
 II & $1$&1& -0.5& 0.05& (a)0.5/(b) 0.05&0.01&  &  \eqref{PRT2}&0.01 &\eqref{PRT2}&0.1&0.01\\
   & $2$&1&   0   & 0.5&   (a)0.5/(b) 0.15&0.01&  &                     &      &\eqref{PRT2}&0.1&0.01\\
   & $3$&1&   0.5& 0.95& (a)0.5/(b) 0.15&0.01&  &                     &      &\eqref{PRT2}&0.1&0.01\\\hline
\end{tabular}
\label{Tab1}
\end{table}

\subsection{Test I: Boltzmann game and Boltzmann control}
In order to validate the present differential game setting and emphasize the differences with respect to a pure control setting, as the one studied in \cite{APZa}, we compare the evolution of two opposite populations of leaders for system \eqref{eq:system_Q}. Therefore, we consider a symmetric configuration where leaders have the same type of strategies, $\psi^k=\mu^k=0.5$, with $k=\{1,2\}$ but with opposite desired opinions ${\bar v}^1=-{\bar v}^2 = 0.5$. We compare the dynamics of the Boltzmann game with the case of single control strategies applied to each leaders' population $u^k$ as in \cite{APZa}. We consider a uniform interaction kernel for the leaders, i.e. $S^k\equiv 1$, whereas for the followers a bounded confidence-type of interaction (for both follower-follower and follower-leader interactions), with the following choices
\begin{align}\label{PRT1}
P(w,w_*) = \chi(|w-w_*|<0.75),\qquad R^k(w,v) = \chi(|w-v|<0.75).
\end{align}
We summarize all computational parameters in Table \ref{Tab1}. At time zero the followers' population is uniformly distributed in the opinion space $I$, whereas the leaders are concentrated close to their opposite desired opinion,  $-\vk$. In Figure \ref{Fig1} we depict in the $I\times [0,T]$ frame the evolution of the densities, respectively, in the top row the leaders' density and in the bottom row the followers' density. Left and right columns corresponds respectively to the dynamics in the Boltzmann game, and in the Boltzmann control setting.
The evolutions show that in both cases followers  are steered to an equilibrium position between the two desired states ${\bar v}^k$, but with different intensities. In absence of leaders' competition,  on the right, the two leaders' populations switch position, since they start from a unpleasant configuration with respect to the desired one, on the other hand, on the left, the initial configuration is preserved due to the balancing effect of the control.
\begin{figure}[t]
\centering
%\subfigure[]{
\includegraphics[scale=0.28]{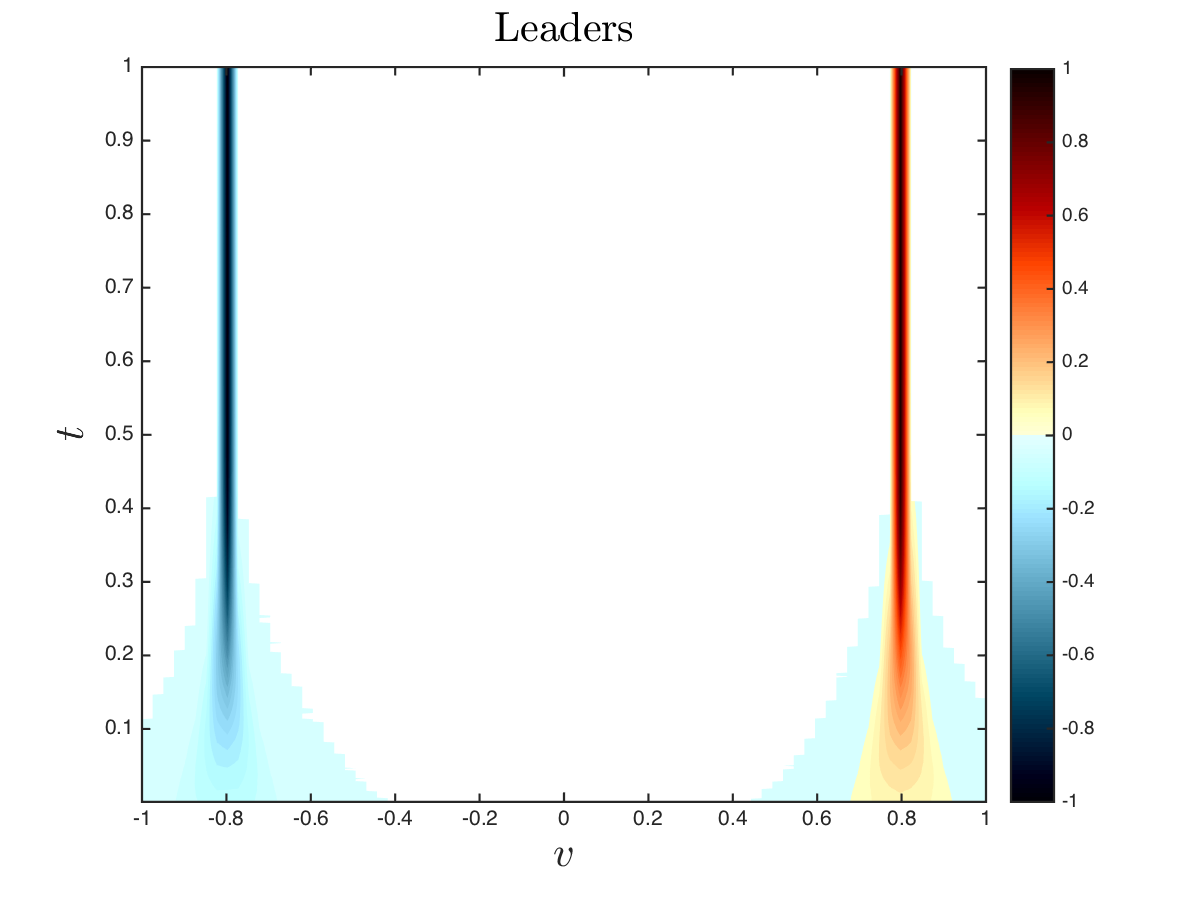}
%}
%\subfigure[]{
\includegraphics[scale=0.28]{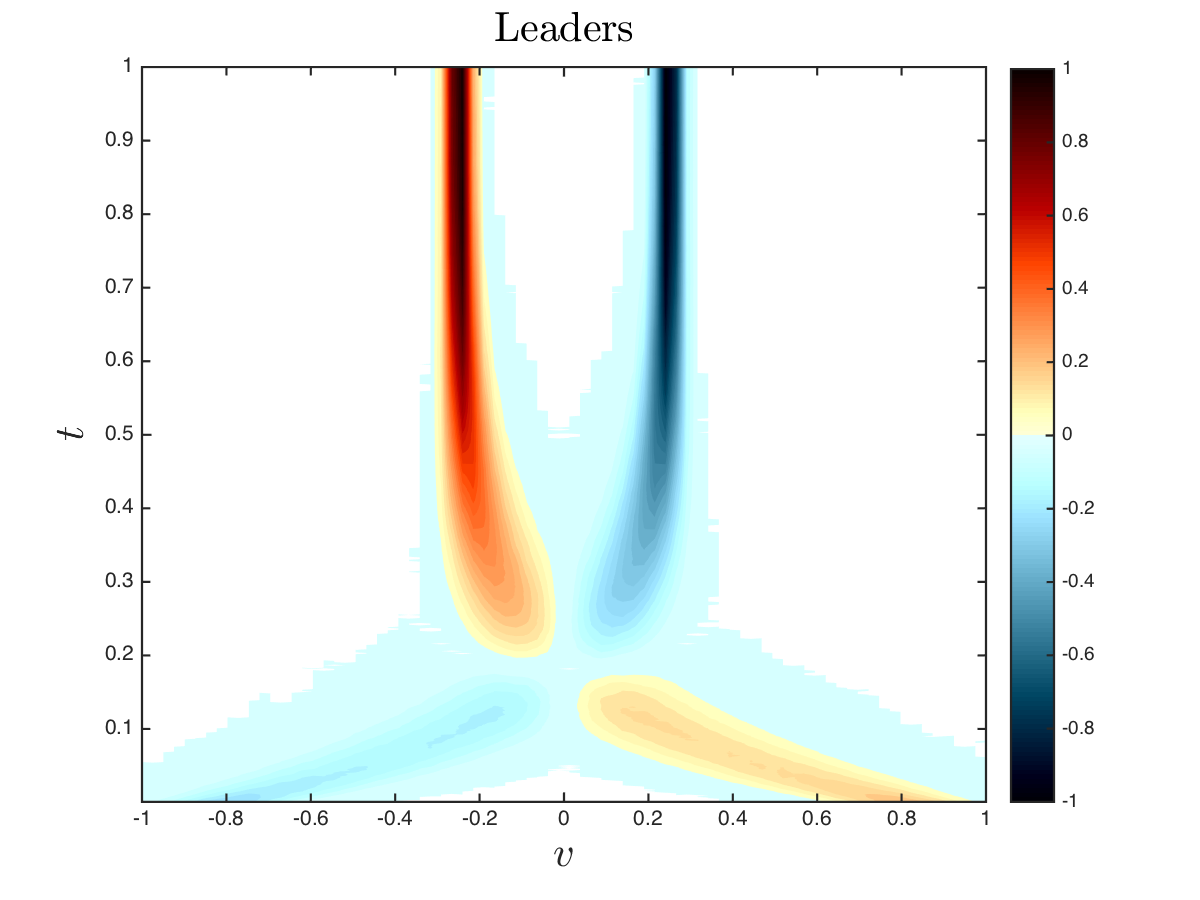}
%}
\\
\subfigure[Boltzmann game: $\sum_{\ell=1}^M u^\ell$]{
\includegraphics[scale=0.28]{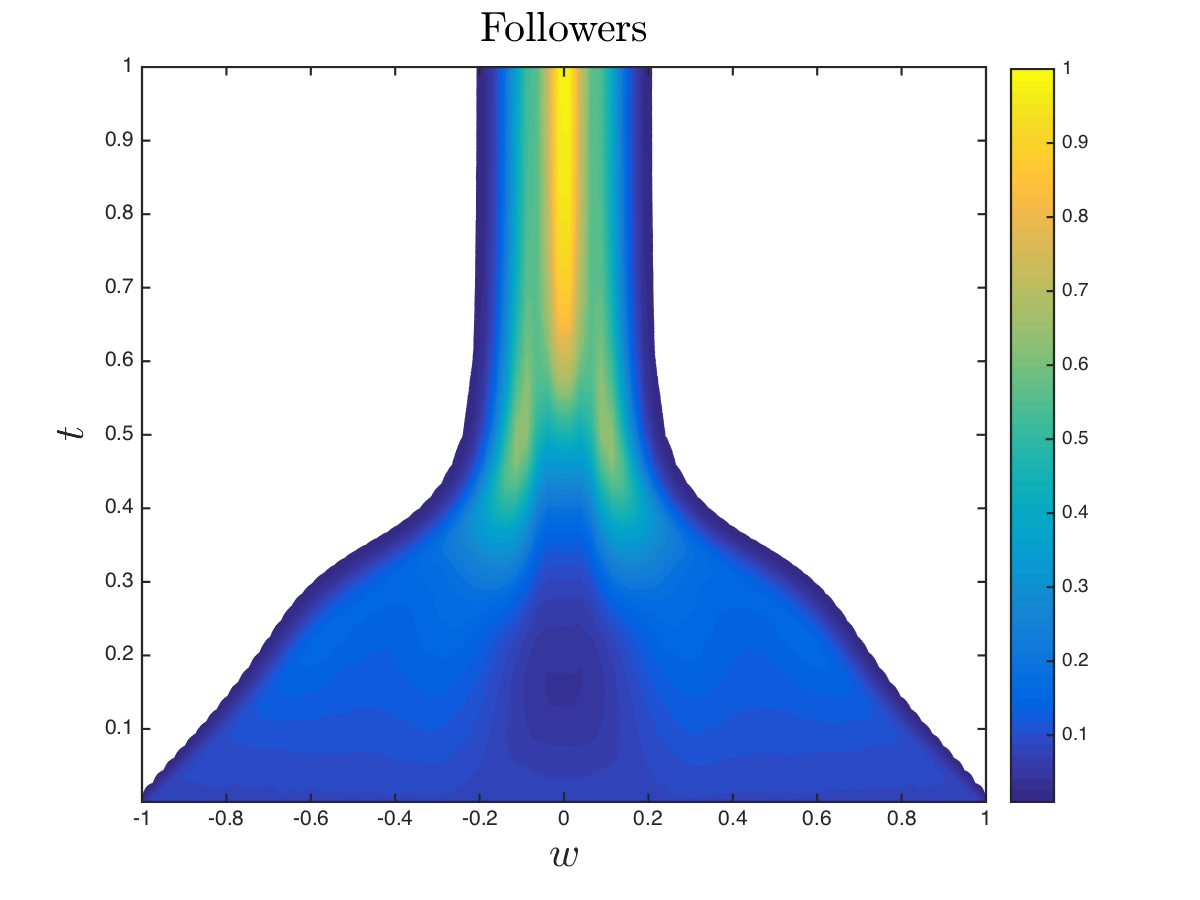}
}
\subfigure[Boltzmann control: $u^k$]{
\includegraphics[scale=0.28]{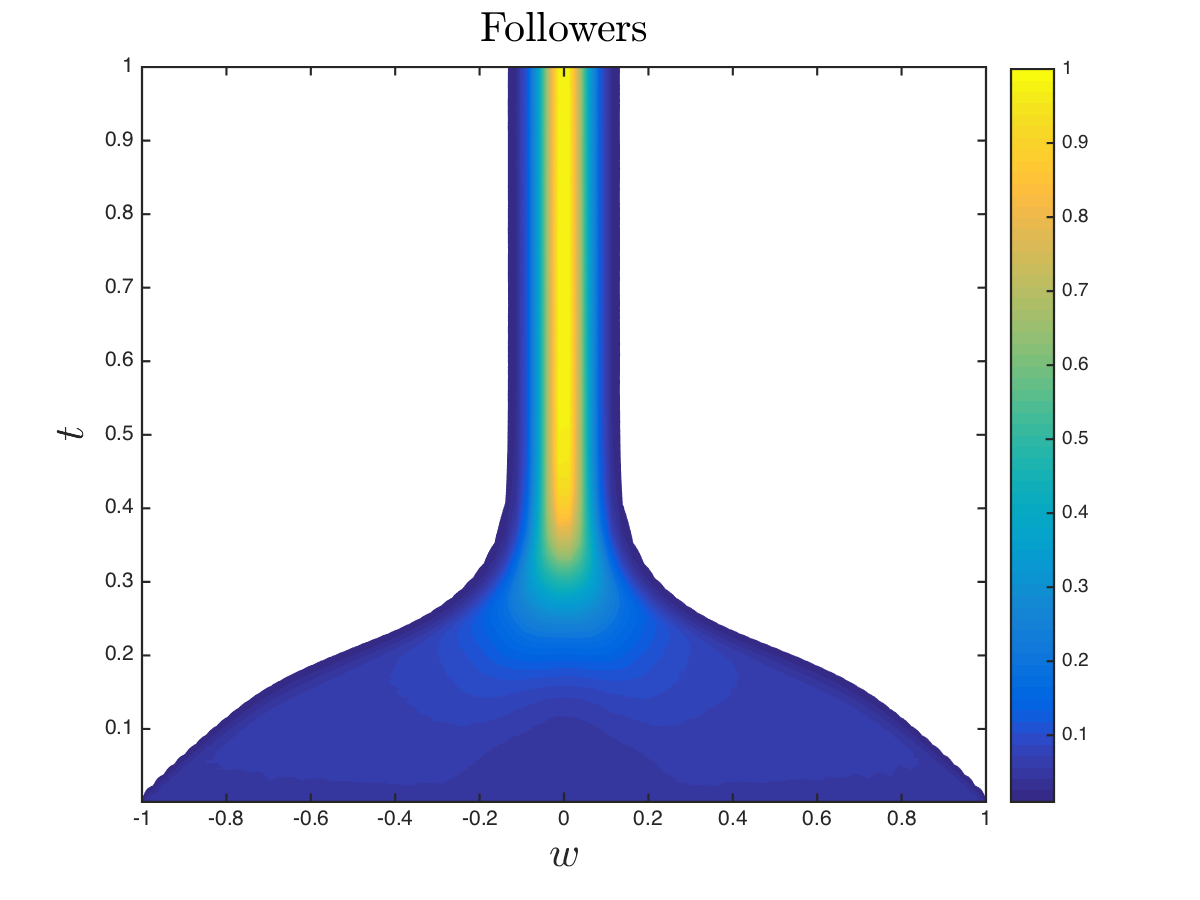}
}
%\subfigure[$t=150$]{\includegraphics[scale=0.225]{figure/T1/F0T_t1_mixedcontrol.png}}
\caption{Test I. Left column (a) shows the present Boltzmann game setting; right column (b) represents the evolution under a Boltzmann control dynamics. Top row represents the evolution of the two leaders' population densities,  bottom row shows the evolution of the followers' density under the influence of the leaders' strategies.}\label{Fig1}
\end{figure}

\subsection{Test II: Different strategies and resources}
We consider now the case of three leaders' populations with different strategies, and we compare two cases for system \eqref{eq:system_Q}: uniform and not-uniform availability of resources. 
Leaders' populations have three different strategies, the first one has target opinion $\bar v^1=-0.5$, and  a populistic approach, namely $\psi^1 = 0.05$, the second population has target opinion $\bar v^2 = 0$ and a balanced strategy with $\psi^2=0.5$, finally the last population has target opinion $\bar v^3=0.5$, with $\psi^3=0.95$ which represents a radical behaviour.  Similar to the previous test we assume a uniform interaction kernel for the leaders, i.e. $S^k\equiv 1$, whereas followers account a bounded confidence-type of interactions, with more restrictive interaction conditions with respect to \eqref{PRT1}
\begin{align}\label{PRT2}
P(w,w_*) = \chi(|w-w_*|<0.25),\qquad R^k(w,v) = \chi(|w-v|<0.25).
\end{align}
The initial density of followers, $f^0(w)$ is uniformly distributed on the opinion interval $[0,0.75]$, and the initial leaders' densities  are centered on their target opinion with normal distribution and standard deviation $\sigma^k =0.1$, as depicted in the first row of Figure \ref{Fig2}. In the second row we report the final time of two different situations, on the left hand-side we observe equal penalization of the control $\nu^1=\nu^2=\nu^3=0.5$, on the right hand-side the control is differently penalized, for populistic leaders, $L^1$, $\nu^1=0.05$, whereas moderate and radical leaders have penalization parameter $\nu^2=\nu^3=0.15$. These situations represent respectively the application of the control with uniform availability of resources, and the case where populistic strategies exhibit more strength in the application of their strategy. In the first case we observe in Figure \ref{Fig2} that followers' density concentrates around the moderate leaders $L^2$, in the second case their density split in two parts where a large part is centered between moderate and populistic leaders, and a tiny percentage collocates between  moderate and radical. 
It is interesting to observe that populist, in the second situation, are able to move a consistent part of the followers density towards their position, having at disposal larger resources. Finally, in Figure \ref{Fig3} we report the evolution of the leaders' densities in the first row, in the space-time frame,  $[-1,1]\times[0,T]$, and of the followers' density in the second row.

\begin{figure}[t]
\centering
\includegraphics[scale=0.3]{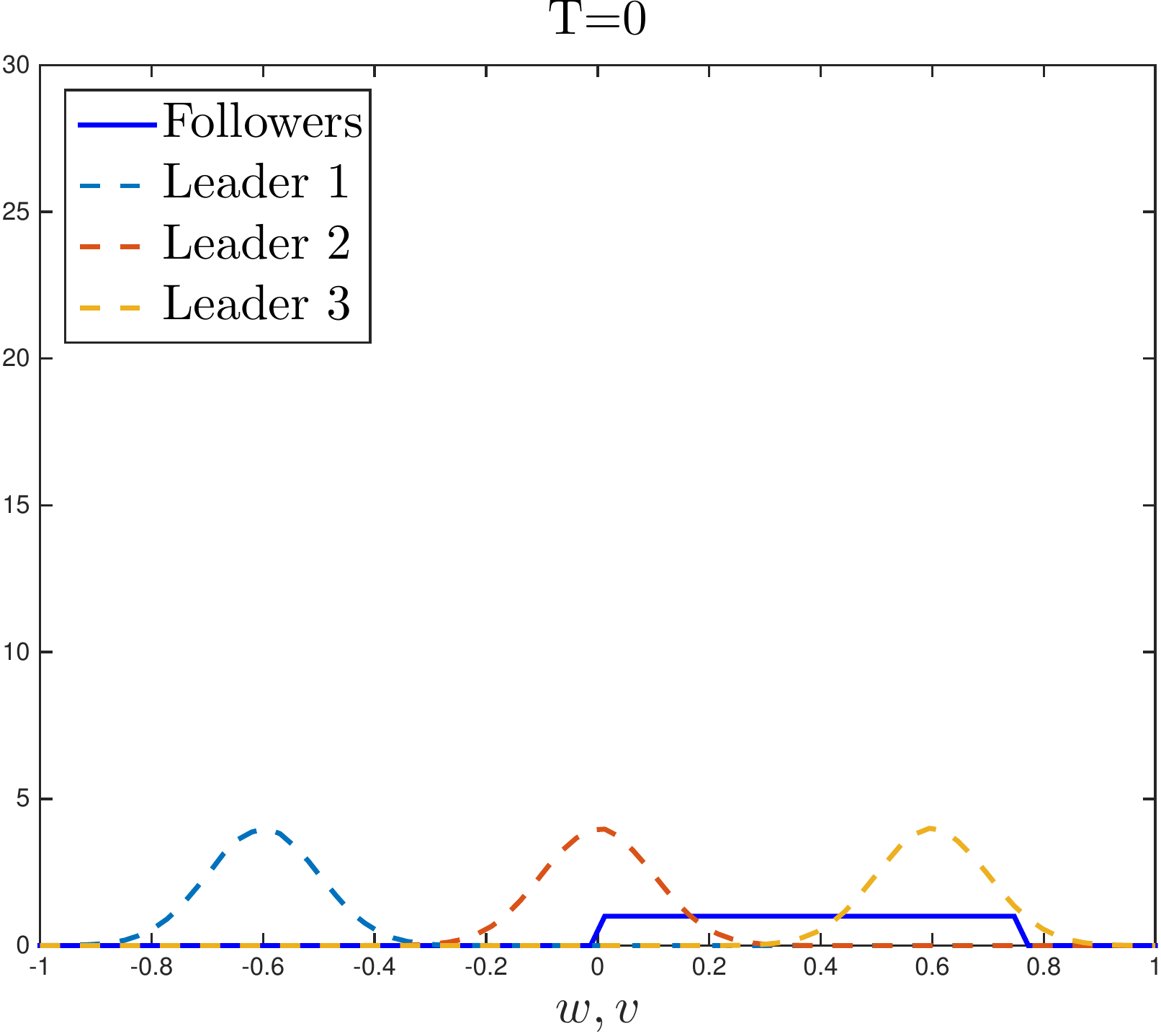}
\\
\subfigure[]{
\includegraphics[scale=0.3]{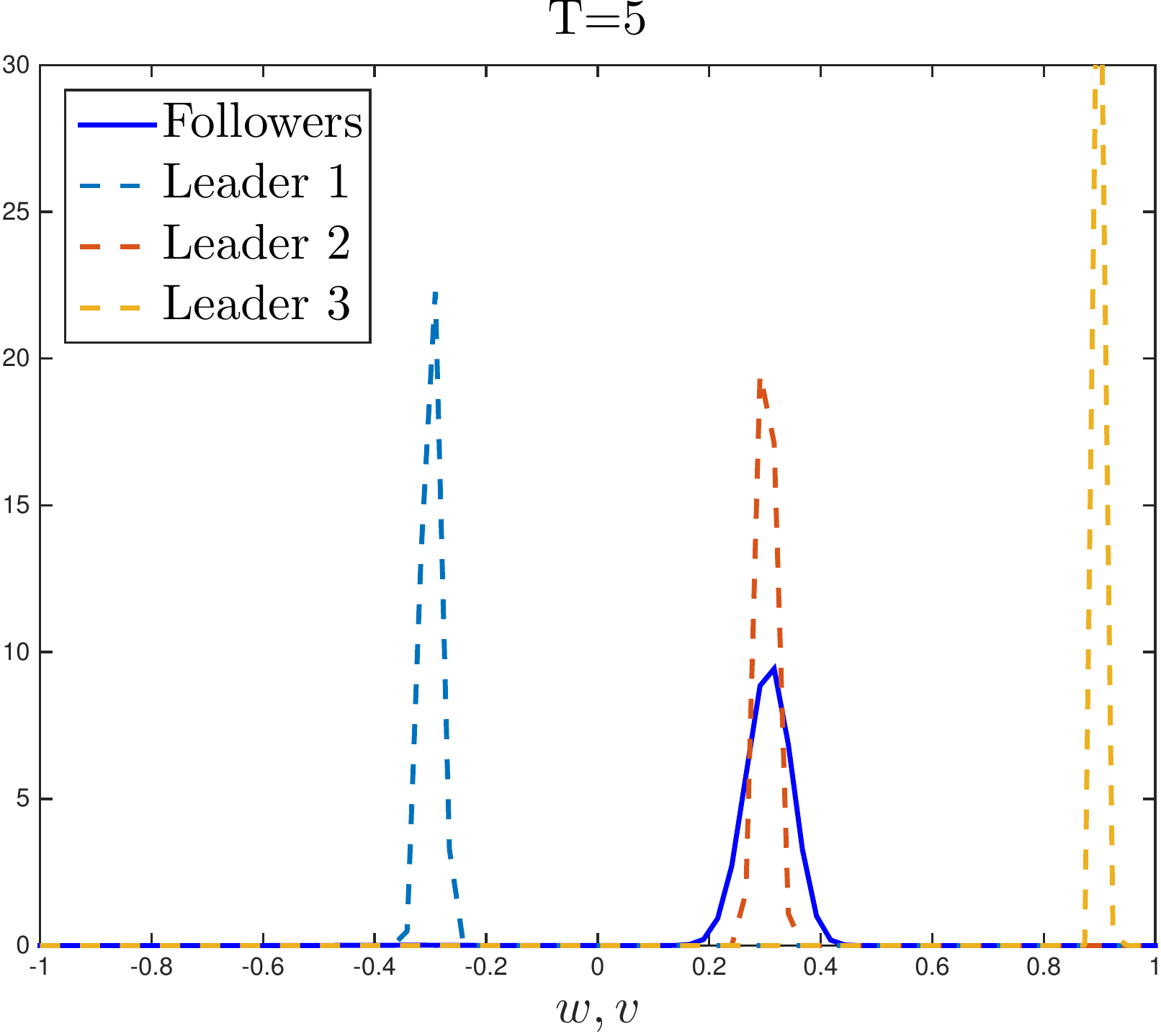}
}
\hspace{+0.5cm}
\subfigure[]{
\includegraphics[scale=0.3]{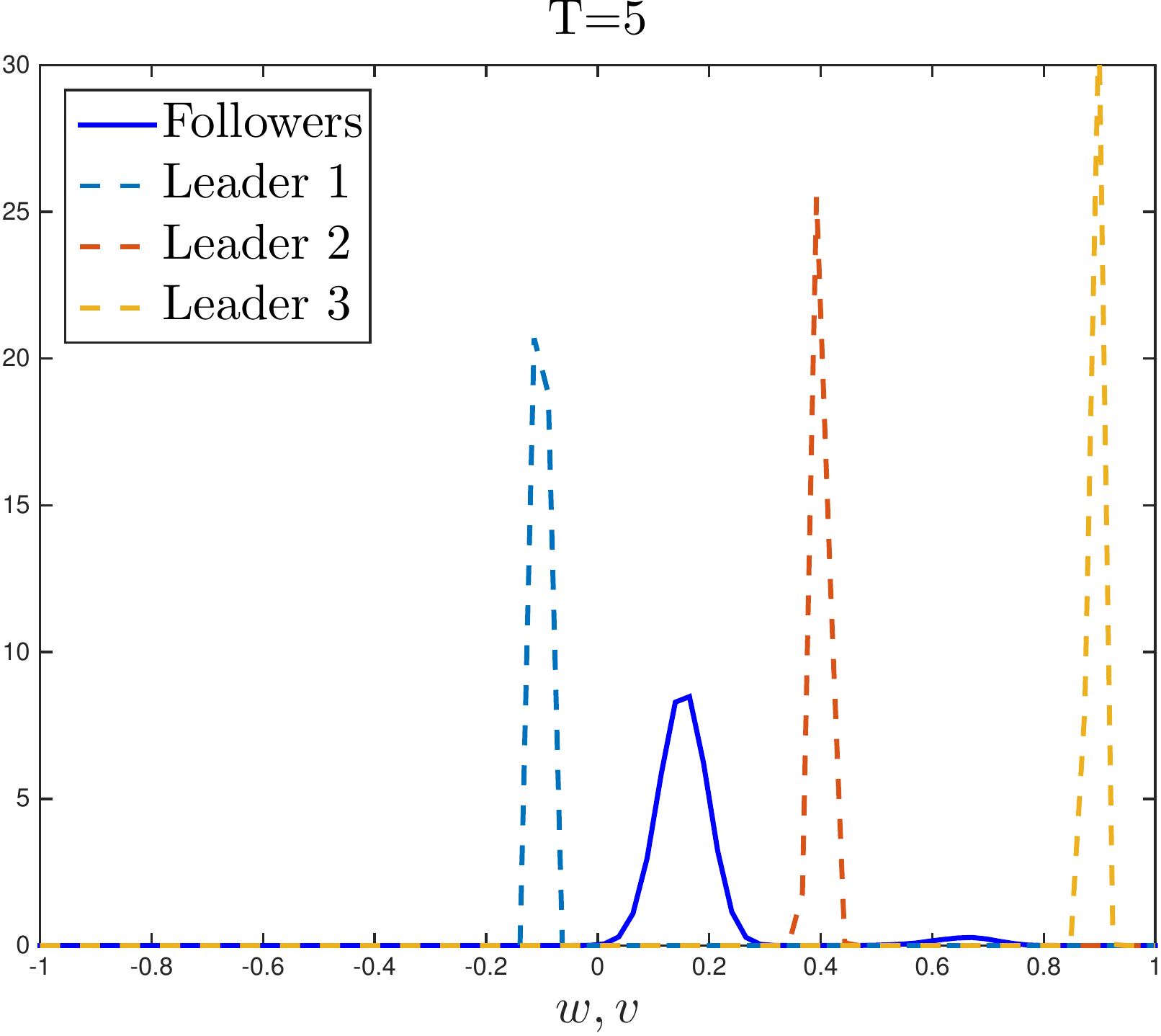}
}
\caption{Test II. First row shows the initial densities of followers and leaders at time $t=0$. Second row shows the final state of the evolution, on the left (a) the case with uniform resources, on the right (b) the case with different penalization of the control strategies.}\label{Fig2}
\end{figure}

\begin{figure}[t]
\centering
%\subfigure[]{
\includegraphics[scale=0.28]{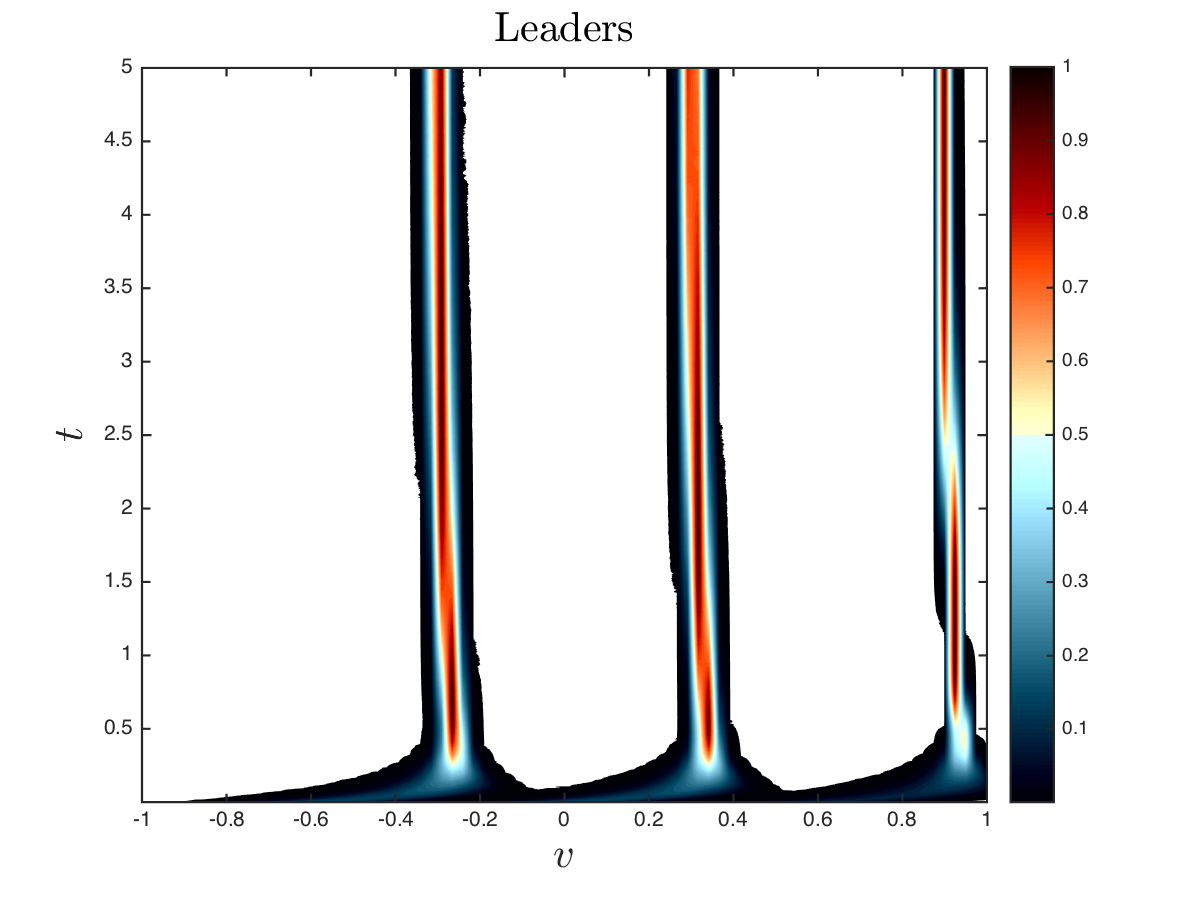}
%}
%\subfigure[]{
\includegraphics[scale=0.28]{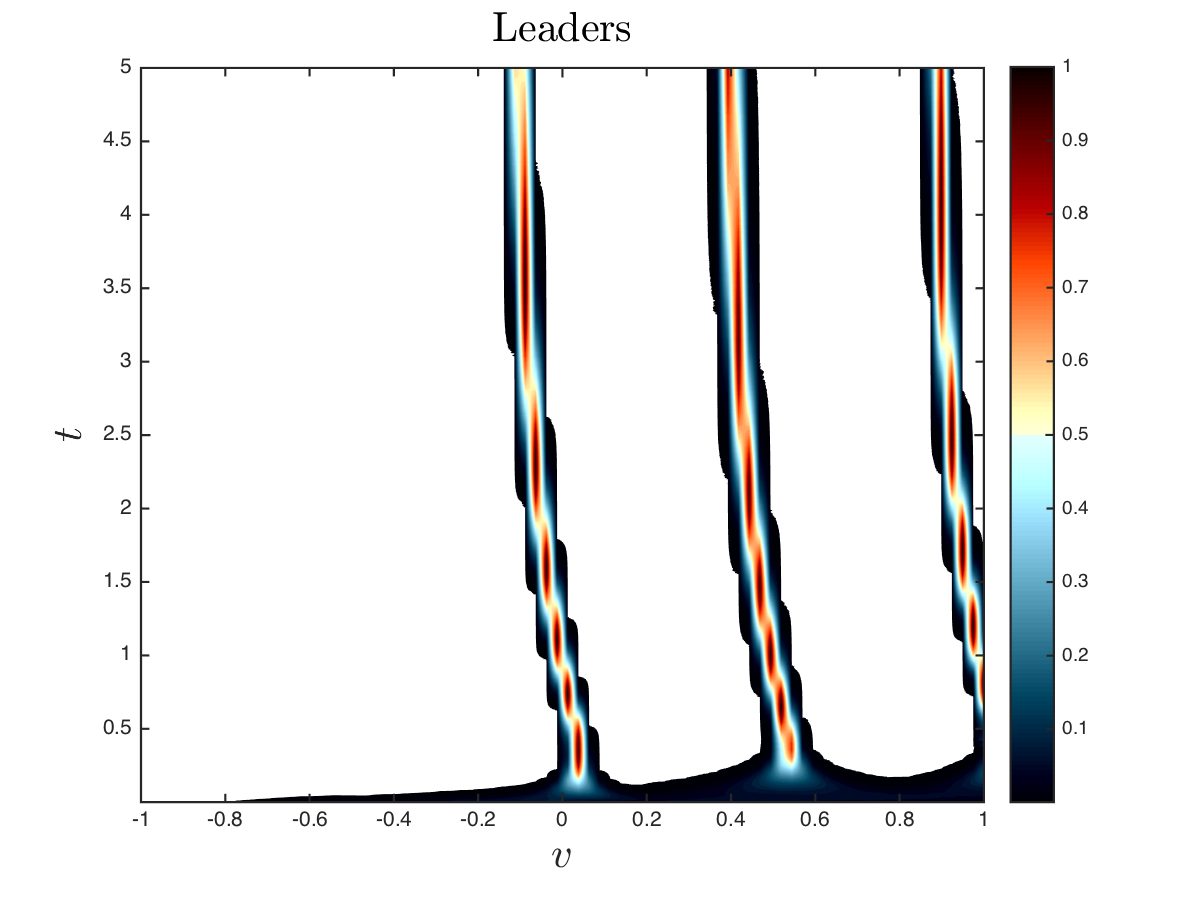}
%}
\\
\subfigure[]{
\includegraphics[scale=0.28]{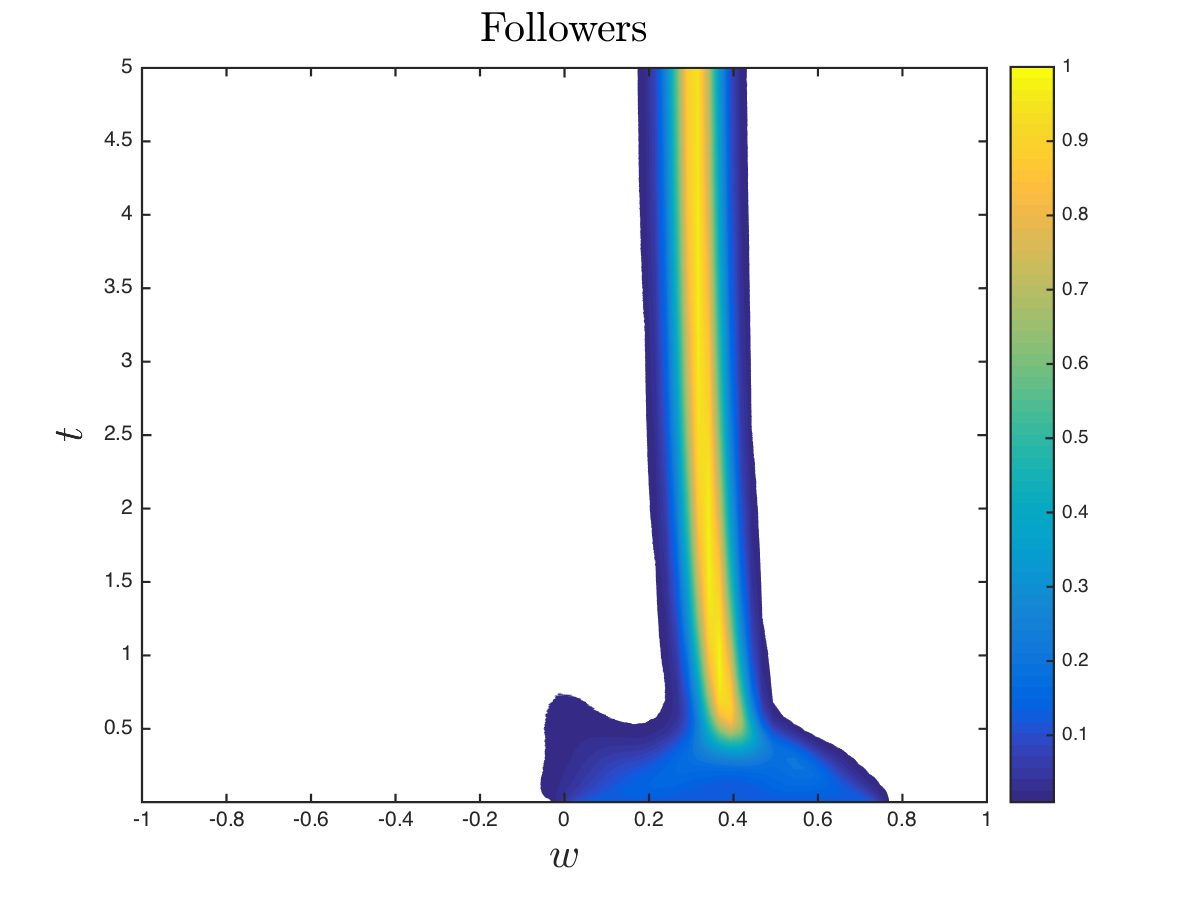}
}
\subfigure[]{
\includegraphics[scale=0.28]{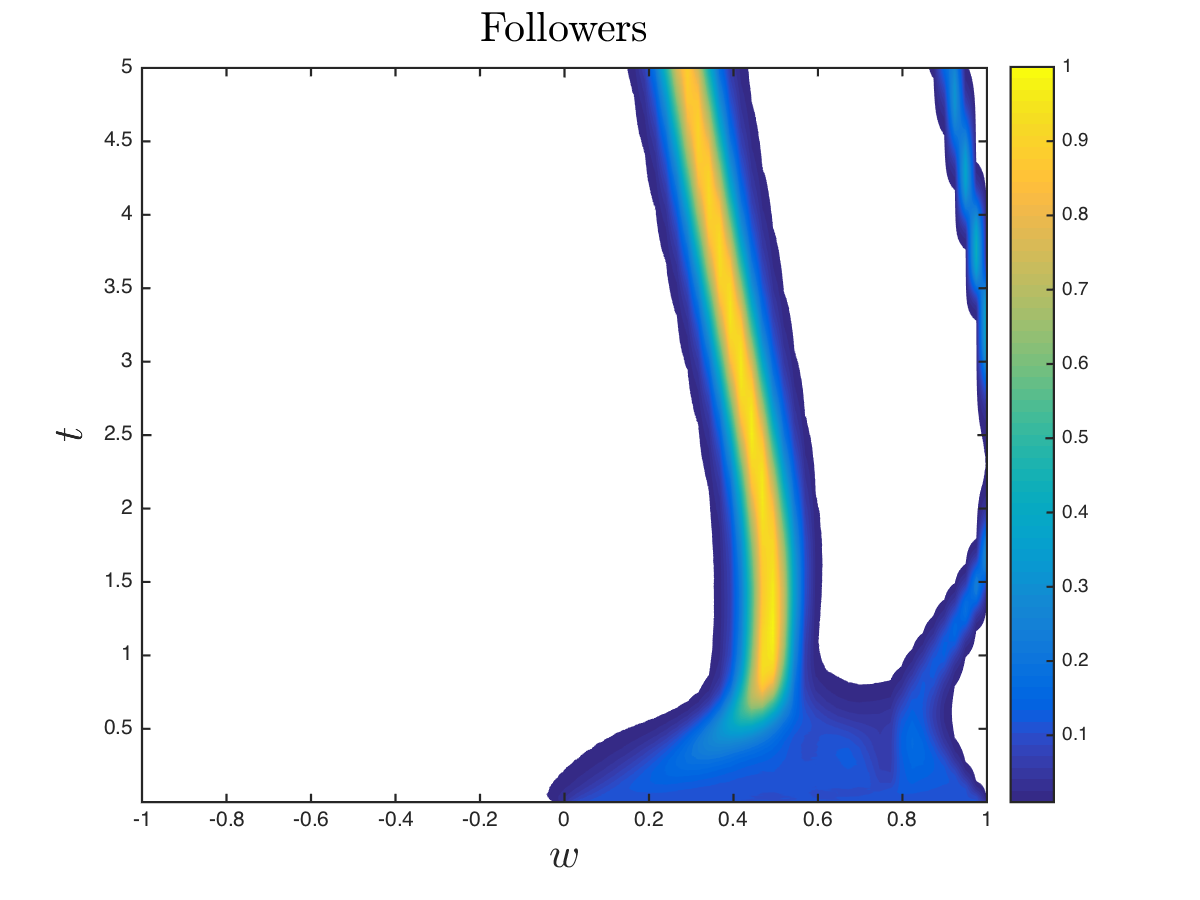}
}
\caption{Test II. Plot of the leaders' and followers' densities in the $(w,t)$-space.on the left (a) the case with uniform resources, on the right (b) the case with different penalization of the control strategies.}\label{Fig3}
\end{figure}
\newpage
\subsection{Test III: The heterogeneous case}
In this latter test we consider the case of heterogeneous followers' population, described by their opinion $w\in[-1,1]$ and an additional competence variable $x\in[0,+\infty]$. This latter variable acts in the way the followers interact with the leaders. In particular we consider the following interaction kernel for the followers-type interactions,
\begin{align}\label{PRT2}
P(w,w_*,x,x_*) = K(x,x_*),\qquad R^k(w,x,v^k) = K(x,\Psi(v^k)),
\end{align}
where the function $K(\cdot,\cdot)$ is defined according to \eqref{eq:def_cMC} with $a=50$,  and the credibility function $\Psi(\cdot)$ according to \eqref{eq:def_R} with parameters,  $\gamma = 0.75,\varsigma = 0.001$. We analyze a constrained dynamics where two populations of leaders compete in order to attract an heterogeneous followers' population. As initial condition we consider the case of uniform opinion distribution in the interval $[-1,1]$ and low knowledge, which evolves through interactions with other agents and through a background $z\sim \mathcal U([0,10])$. Furthermore, we consider in the knowledge dynamics $\lambda = 5\cdot10^{-3}$, $\lambda_B = 5\cdot 10^{-3}$, $\lambda = \lambda_B+\lambda$, $\sigma_{\kappa}=2.5\cdot 10^{-3}$. 
The detailed parameters for leaders' populations used in the numerical simulation are reported in Table \ref{Tab2}.

In Figure \ref{Fig3}--\ref{Fig4} we can observe how the leader with a stronger populist attitude is capable to drive the opinion of the followers with lower competence due to their low credibility level. On the other hand, the leader with a stronger radical attitude has a larger influence over the highly skilled followers. 

\begin{table}[t]
\centering
\caption{Computational parameters for Test III.}
\begin{tabular}{c|c|c|cccc|c|cc|ccc}
\hline
\hline
Test & $L_k$ & $S^k(\cdot)$ &$\vk$& $\psi^k$&$\nu^k$ &$\sigma_{\eta^k}$ &$F$&$P(\cdot)$& $\sigma_{\eta}$ &$R^k(\cdot)$&$c_{FL}^k$&$\sigma_{\eta^k}$\\
\hline
\hline
 III & $1$&1&    0.50& 0.1 & 0.5 &0.01& &\eqref{PRT2}&0.01&\eqref{PRT2}&0.1&0.01\\
    & $2$&1& -0.5& 0.75 & 0.1 &0.01&  &                     &      &\eqref{PRT2}&0.1&0.01\\
\hline
\end{tabular}
\label{Tab2}
\end{table}

\begin{figure}[t]
\centering
\subfigure[$T=0$]{\includegraphics[scale=0.4]{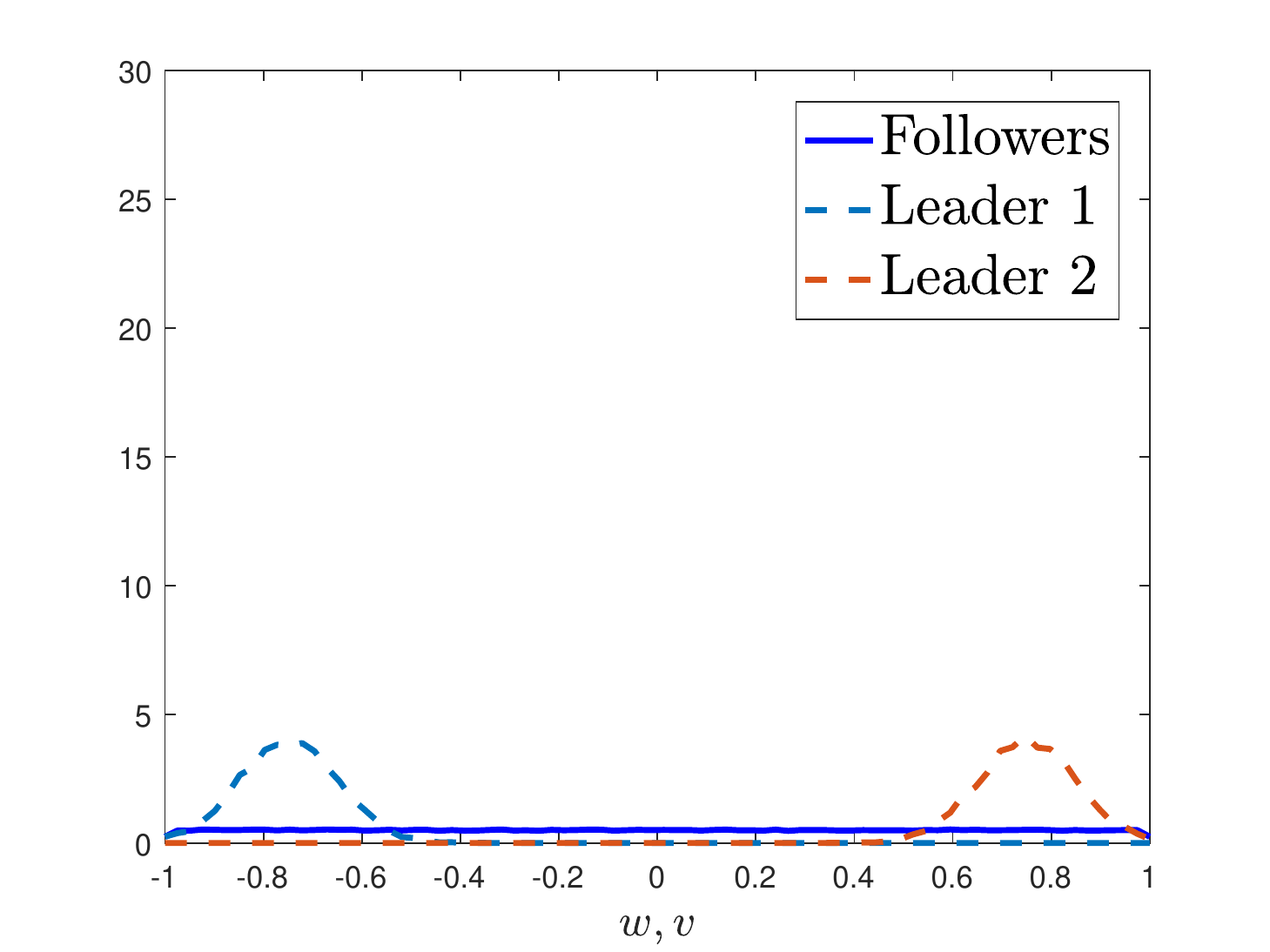}}
\subfigure[$T=1$]{\includegraphics[scale=0.4]{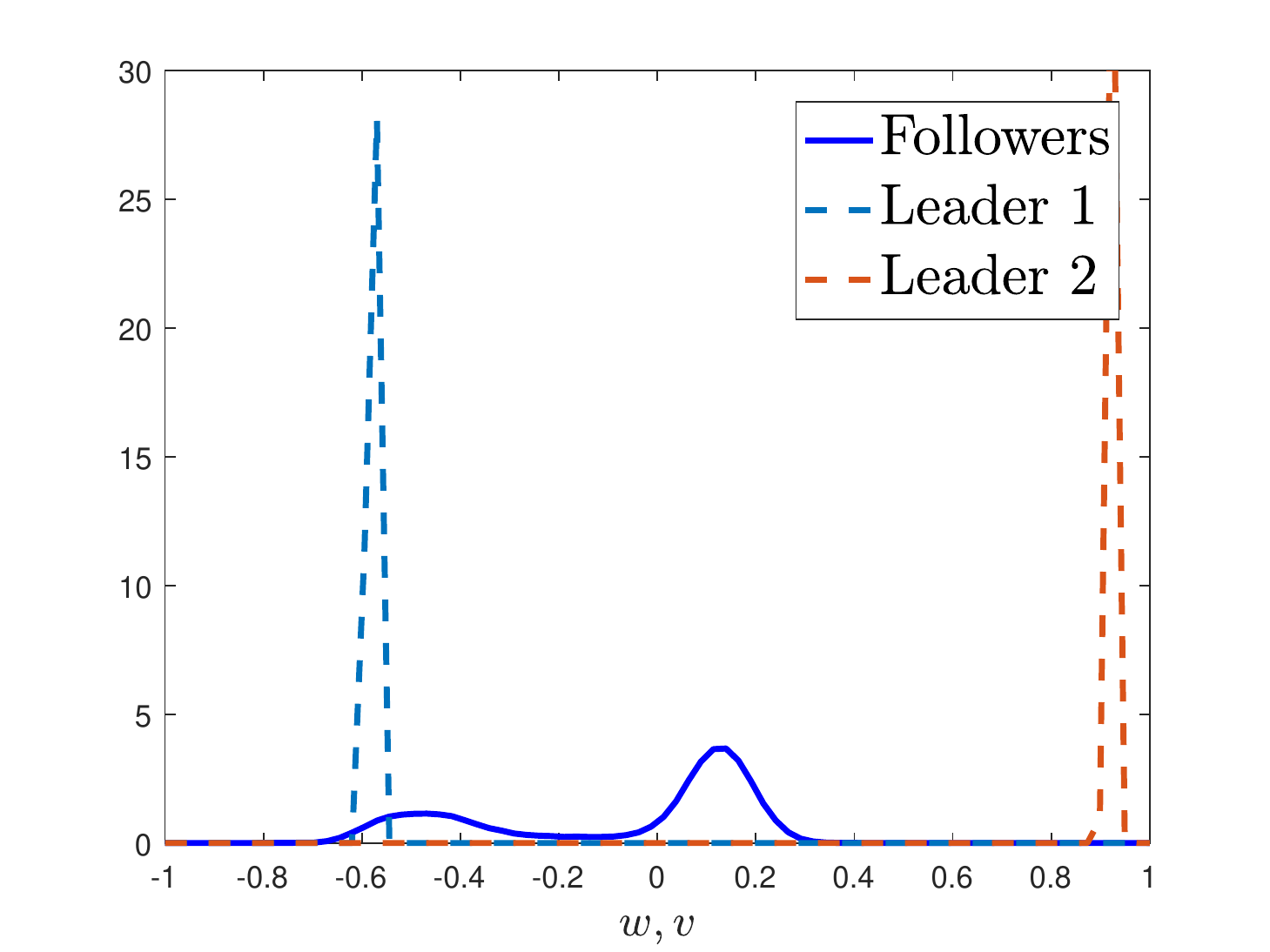}}\\
\subfigure[$T=5$]{\includegraphics[scale=0.4]{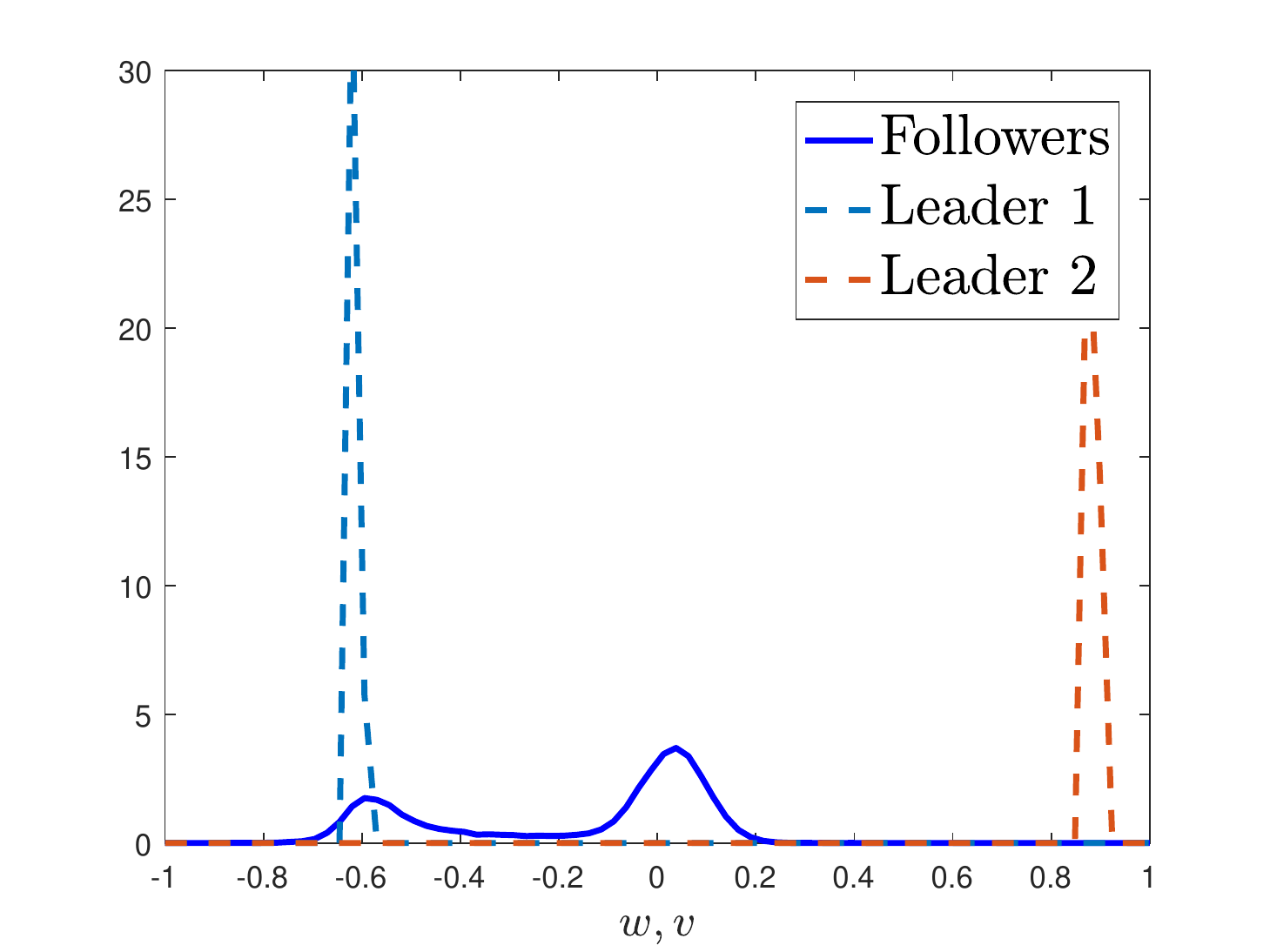}}
\subfigure[$T=10$]{\includegraphics[scale=0.4]{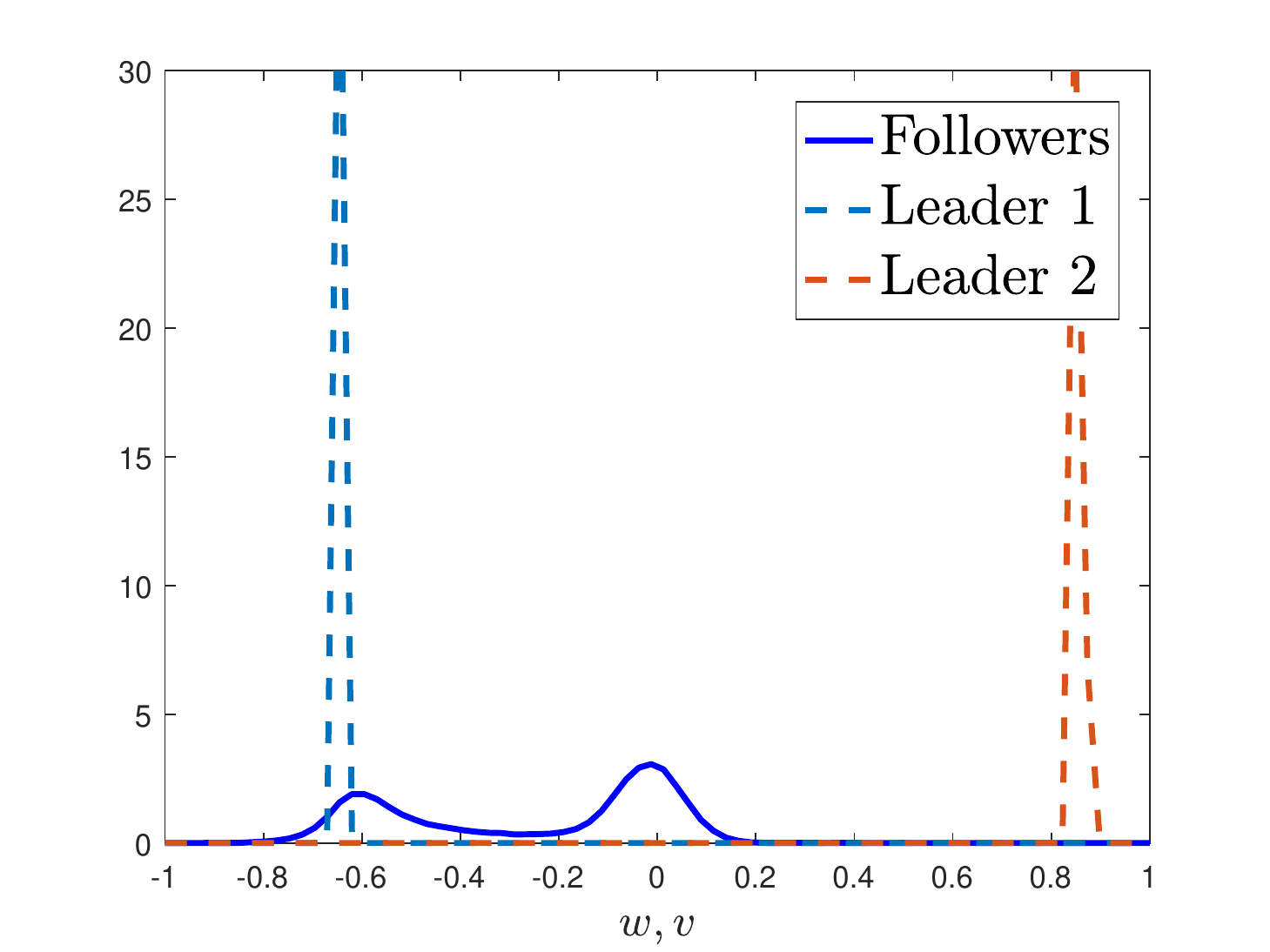}}
\caption{Test III. From left to right: evolution of the leaders-followers systems in the case of knowledge-based interactions. Leader 2's population has a radical strategy, $\psi^2=0.75$, whereas Leader 1's population has a populist attitude, $\psi^1 = 0.1$.}\label{Fig3}
\end{figure}

\begin{figure}[t]
\centering
\subfigure[$T=0$]{\includegraphics[scale=0.3]{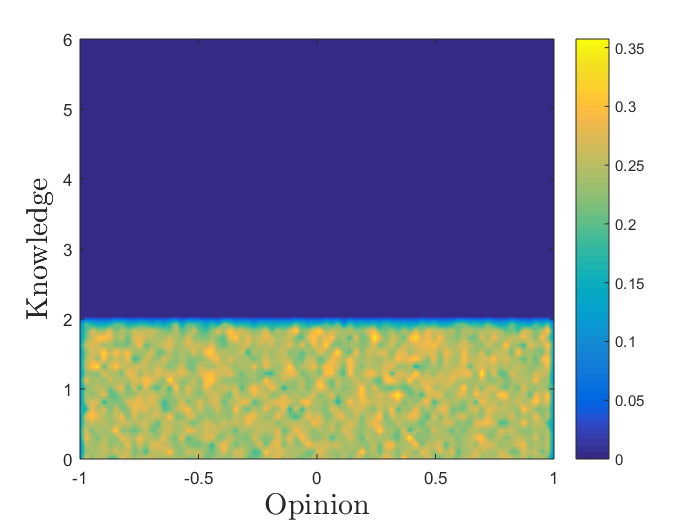}}
\subfigure[$T=1$]{\includegraphics[scale=0.3]{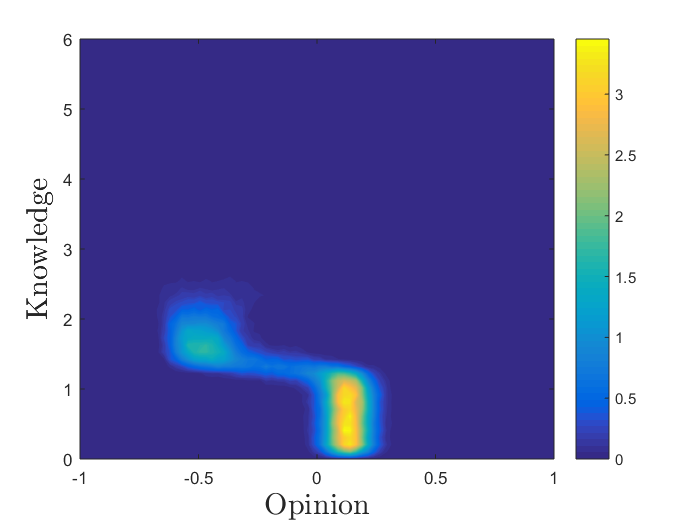}}\\
\subfigure[$T=5$]{\includegraphics[scale=0.3]{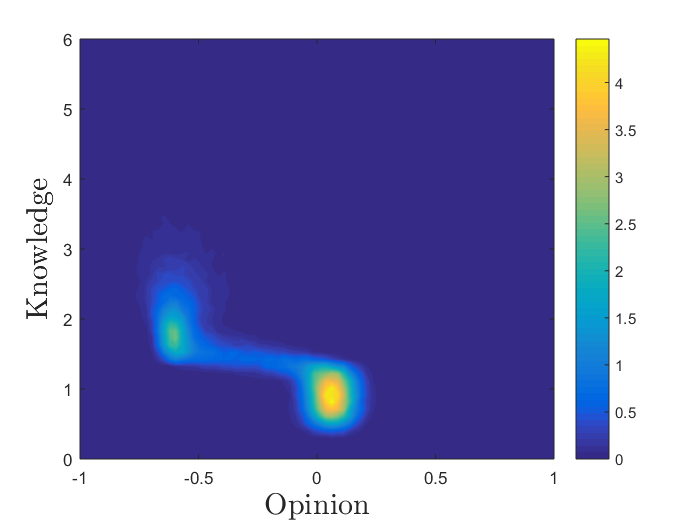}}
\subfigure[$T=10$]{\includegraphics[scale=0.3]{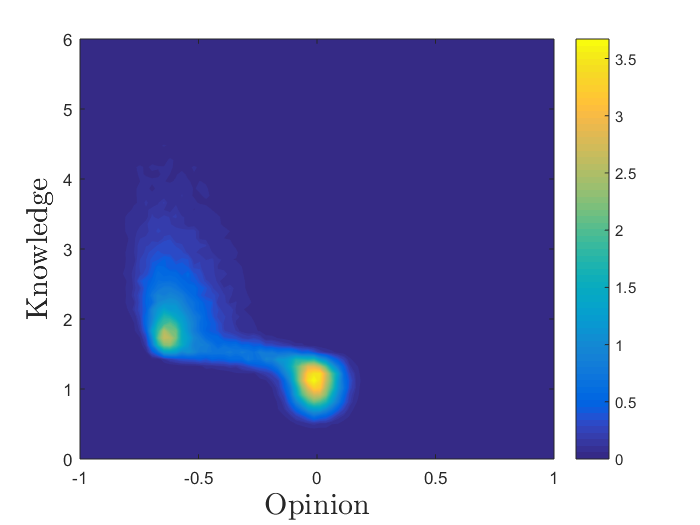}}
\caption{Test III. Evolution of the followers' density, $f(x,w,t)$, in the knowledge-opinion space. The uniform followers' density evolves in time towards a final state, which is biased towards the more credible leaders' population for higher level of knowledge, whereas for lower level of knowledge the populist and less credible leader is more attractive. 
}\label{Fig4}
\end{figure}

\section*{Conclusion}
A differential game involving different groups of leaders, each one with its own strategy, has been considered and studied in the case of a population of heterogeneous followers. The approach used is based on a model predictive approximation of the leaders' game once it has been reduced to binary interactions. This permits to obtain explicitly the best reply strategy for each leader and to write the corresponding Boltzmann system. Approximating the system through a Fokker-Planck equation yields, as usual, analytic stationary solutions for the corresponding opinion distributions. Numerical results show the strong impact of heterogeneity in the outcome of the leaders' competition.

\section*{Acknowledgments}
M. Z. acknowledges the "Compagnia di San Paolo" (Torino, Italy) for the support.

%%%%%%%%%% Insert bibliography here %%%%%%%%%%%%%%

\end{document}